\documentclass[a4paper,11pt,oneside,reqno]{amsart}

\usepackage{hyperref}
\hypersetup{colorlinks=true,allcolors=blue}
\usepackage{hypcap}

\usepackage[leqno]{amsmath}
\usepackage{amsfonts,amssymb,amsthm}
\usepackage{mathrsfs}
\usepackage{esint}
\usepackage{mathtools}
\usepackage{graphicx} 
\usepackage{galois}
\usepackage[utf8]{inputenc}
\usepackage{enumitem}
\usepackage[margin=1.2in]{geometry}
\usepackage{tikz}
\usepackage{pgfplots}

\theoremstyle{plain}
\newtheorem{thm}{\protect\theoremname}[section]
\theoremstyle{plain}
\newtheorem{prop}[thm]{\protect\propositionname}
\theoremstyle{definition}
\newtheorem{defn}[thm]{\protect\definitionname}
\theoremstyle{plain}
\newtheorem{cor}[thm]{\protect\corollaryname}
\theoremstyle{remark}
\newtheorem{rem}[thm]{\protect\remarkname}
\newtheorem{ex}[thm]{Example}
\theoremstyle{plain}
\newtheorem{lem}[thm]{\protect\lemmaname}

\providecommand{\corollaryname}{Corollary}
\providecommand{\definitionname}{Definition}
\providecommand{\lemmaname}{Lemma}
\providecommand{\propositionname}{Proposition}
\providecommand{\remarkname}{Remark}
\providecommand{\theoremname}{Theorem}

\numberwithin{equation}{section}
\numberwithin{figure}{section}
\numberwithin{table}{section}

\newcommand{\h}{\hat{h}}

\renewcommand{\L}{\mathcal{L}}
\newcommand{\R}{\mathbb{R}}
\newcommand{\N}{\mathbb{N}}
\newcommand{\C}{\mathbb{C}}
\renewcommand{\S}{\mathbb{S}}
\newcommand{\Z}{\mathbb{Z}}

\renewcommand{\Re}{\mathrm{Re}}

\newcommand{\supp}{\mathrm{supp}\,}
\newcommand{\Int}{\mathrm{Int}\,}
\newcommand{\spec}{\mathrm{Spec}}
\newcommand{\Id}{\mathrm{Id}}
\newcommand{\loc}{\mathrm{loc}}
\newcommand{\Lip}{\mathrm{Lip}}

\newcommand{\pt}{\partial_{t}}

\renewcommand{\d}{\mathrm{d}}
\newcommand{\dx}{\d x}

\newcommand{\dt}{\d t}
\newcommand{\ds}{\d s}

\newcommand{\dz}{\d z}

\begin{document}

\title{Stabilization of Damped Waves on Spheres
and Zoll Surfaces of Revolution}
\author{Hui Zhu}
\thanks{The original publication is available at www.esaim-cocv.org}

\begin{abstract}
We study the strong stabilization of wave equations on some sphere-like manifolds, with rough damping terms which do not satisfy the geometric control condition posed by Rauch-Taylor~\cite{R-T} and Bardos-Lebeau-Rauch~\cite{B-L-R}. We begin with an unpublished result of G.~Lebeau, which states that on~$ \S^d $, the indicator function of the upper hemisphere strongly stabilizes the damped wave equation, even though the equators, which are geodesics contained in the boundary of the upper hemisphere, do not enter the damping region. Then we extend this result on dimension~2, to Zoll surfaces of revolution, whose geometry is similar to that of~$ \S^2 $. In particular, geometric objects such as the equator, and the hemi-surfaces are well defined. Our result states that the indicator function of the upper hemi-surface strongly stabilizes the damped wave equation, even though the equator, as a geodesic, does not enter the upper hemi-surface either.
\end{abstract}

\maketitle



\section{Introduction}

\subsection{Problem of Stabilization and Main Result}

Consider the Cauchy problem of the damped wave equation on a compact Riemannian manifold $(M,g)$ without boundary.
\begin{equation}
\label{eq:Damped-Wave-Equation}
\begin{cases}
(\pt^{2}-\Delta+a\pt) u=0 & \mathrm{in\ }\mathcal{D}'(\R\times M),\\
(u,\pt u)_{t=0}=(u_{0},u_{1}) & \in H^{1}(M)\times L^{2}(M).
\end{cases}
\end{equation}
Here $\Delta=\Delta{g}$ is the Laplace-Beltrami operator with respect to the metric~$g$. The function~$a\in L^{\infty}(M)$ is non-negative, and $a\pt u$ is called the damping term, as it causes decay in energy (defined below). There is a unique solution $ u \in C^{1}(\R,L^{2}(M))\cap C(\R,H^{1}(M))$ to~\eqref{eq:Damped-Wave-Equation} by the theorem of Hille-Yosida. The energy defined by
\begin{equation}
\label{eq:def-energy}
E(u,T)=\frac{1}{2}\|{\nabla u(T)}\|_{L^{2}(M)}^{2}+\frac{1}{2}\|{\pt u(T)}\|_{L^{2}(M)}^{2}
\end{equation}
decays monotonically as $T$ increases, due to the non-negativity of $a$
and the identity
\begin{equation}
\label{eq:energy-identity}
E(u,T)=E(u,0)-\int_{0}^{T}\int_{M}a(x)|\pt u(t,x)|^{2}\,\dx\,\dt.
\end{equation}
A natural question to ask is whether, as a consequence of the damping effect, 
\begin{equation*}
\lim_{T\to+\infty} E(u,T) = 0 
\end{equation*}
for every solution~$ u $ to~\eqref{eq:Damped-Wave-Equation}. If this is true, we say that~$ a $ \textit{weakly stabilizes}~\eqref{eq:Damped-Wave-Equation}. When such a stabilization is uniform for all solutions, or more precisely, if for some function $f:\R_{\ge 0}\to\R_{\ge 0}$ with $\lim_{T\to+\infty} f(T)=0$ and every solution~$u$ to~\eqref{eq:Damped-Wave-Equation}, we have for all $ T \ge 0 $,
\begin{equation*}
E(u,T)\le E(u,0)\times f(T),
\end{equation*} 
then we say that~$ a $ \textit{strongly stabilizes}~\eqref{eq:Damped-Wave-Equation}. It is well known that whenever the strong stabilization holds, the function~$ f $ could be chosen of the form
\begin{equation*}
f(T) = C e^{-\beta T}, \quad C > 0,\ \beta > 0,
\end{equation*}
so that we have in fact a uniform exponential decay of energy (see for example~\cite{B-G} for an elementary proof).

When $a\in C(M)$, Rauch-Taylor gave in~\cite{R-T} a sufficient condition (a \emph{geometric control condition}, to be stated as condition (2) in the following Theorem~\ref{thm:GCC-and-Energy-Decay-of-DW}) for strong stabilization, followed by Bardos-Lebeau-Rauch~\cite{B-L-R}, who showed that this is in fact an equivalent condition (even for the similar problem of stabilization on manifolds with boundaries, which will not be elaborated here).
\begin{thm}[Bardos-Lebeau-Rauch]
\label{thm:GCC-and-Energy-Decay-of-DW}
Let $(M,g)$ be a compact Riemannian manifold without boundary, and $0\le a\in C(M)$, then the following two statements are equivalent.
\begin{enumerate}
\item $ a $ strongly stabilizes~\eqref{eq:Damped-Wave-Equation};
\item All geodesics of $M$ enter the open set $\{a>0\}$. That is, for $ x \in M $, let $\gamma$ be a geodesic starting from $x$ (i.e.~$\gamma(0)=x$), then for some $ t \ge 0 $, $\gamma(t)\in\{a>0\}$.
\end{enumerate}
\end{thm}
The proof of Theorem~\ref{thm:GCC-and-Energy-Decay-of-DW} in~\cite{B-L-R} used the propagation theorem developed by Melrose-Sj\"{o}strand~\cite{Melrose-Sjostrand}. Lebeau~\cite{Lebeau} managed to use microlocal defect measures (which is due to Gérard~\cite{Gerard-MDM} and Tartar~\cite{Tartar-H-Measures}, see also \cite{Burq-Bourbaki}) and an argument by contradiction to give a new and much simpler proof. However, when $a\in L^{\infty}(M)$, it remains an open problem to give an equivalent condition for strong stabilization, even though the following necessary condition and sufficient condition are known to be classical, which follow by analyzing the proof of Theorem~\ref{thm:GCC-and-Energy-Decay-of-DW}.
\begin{prop}
\label{prop:Necessary-Sufficient-Condition-for-Stability} 
Let $ (M,g) $ be a compact Riemannian manifold without boundary, and let $0 \le a\in L^{\infty}(M)$, 
\begin{enumerate}
\item if $ a $ strongly stabilizes~\eqref{eq:Damped-Wave-Equation}, then all geodesics of~$M$ intersect with~$\supp a$; 
\item if all geodesics of $M$ enter the open set
$ U(a)=\bigcup_{\epsilon>0}\mathrm{Int}\{x:a(x)>\epsilon\}, $
then $ a $ strongly stabilizes~\eqref{eq:Damped-Wave-Equation}.
\end{enumerate}
\end{prop}

When $ a \in C(M) $, condition~(2) is also necessary because in this case $ U(a) = \{a>0\} $, and we conclude by Theorem~\ref{thm:GCC-and-Energy-Decay-of-DW}. However, for general $ a \in L^\infty(M) $, these two conditions are not sharp. Typical examples are as follows. Let $M=\S^{2}=\{x^2+y^2+z^2 = 1\}$, define the equator $ \Gamma = \S^2 \cap \{z=0\} $ and the hemispheres $ \S^2_\pm = \S^2 \cap \{\pm z > 0\} $. Let $0\le a\in C(\S^2)$ be zero exactly on the equator, while~$a>0$ elsewhere. Theorem~\ref{thm:GCC-and-Energy-Decay-of-DW} says that~$ a $ does not strongly stabilize~\eqref{eq:Damped-Wave-Equation}, for the equator~$ \Gamma $, as a geodesic, does not enter $ \{a>0\} = \S^2_+ \cup \S^2_- $, even though all geodesics enter $ \supp a = \S^2 $. On the other hand, let $a=1_{\S_{+}^{2}}$ be the indicator function of the upper hemisphere, then the equator does not enter $U(a)=\Int \S_{+}^{2}$.
However, the following unpublished result due to Gilles Lebeau shows that~$ a $ indeed strongly stabilizes~\eqref{eq:Damped-Wave-Equation}.

\begin{thm}[Lebeau, unpublished]\label{thm:Lebeau}
For $ d \ge 1 $, let $ \S^d = \{x = (x_1,\ldots,x_{d+1}) \in \R^{d+1} : x_1^2 + \cdots x_{d+1}^2= 1 \} $ be the $ d $-dimensional unit sphere, which inherits the Riemannian metric from~$ \R^{d+1} $. Let $ \S^d_+ = \S^d \cap \{x_{d+1} > 0\} $ denote the upper hemisphere, then $a(x)=1_{\S^d_{+}}(x)$ strongly stabilizes~\eqref{eq:Damped-Wave-Equation}.
\end{thm}

We will first give a simple proof of this theorem (see Section~\ref{sec:Stabilization-S^d-Lebeau}) using the spectral distribution of the spherical Laplacian, and the symmetries of spherical harmonics. Then we extend this result, on dimension 2, to Zoll surfaces of revolution.
\begin{defn}
A Zoll manifold is a Riemannian manifold whose geodesic flow is periodic. A Zoll surface of revolution is a 2 dimensional Zoll manifold, on which the group $\S^1$ acts smoothly, faithfully, and isometrically. 
\end{defn}

We refer to Besse~\cite{Besse} for an introduction of Zoll manifold. Some fundamental geometric properties and examples are stated below. In particular $\S^{d}$ ($ d \ge 1 $) are Zoll manifolds, and~$ \S^2 $ is a Zoll surface of revolution. The geometry of a Zoll surface of revolution resembles much that of $ \S^2 $, which makes it natural for the generalization of Theorem~\ref{thm:Lebeau}. (However, on general Zoll manifolds, such resemblance is not yet clear to the author.) Indeed, we will use the following two aspects of resemblance for our generalization.
\begin{itemize}[noitemsep]
\item Local Geometry: On Zoll surfaces of revolution, the geometric objects such as the equator, and the upper and lower hemi-surfaces are well defined. Moreover, the local geometry near the equator is similar to that near a great circle of~$ \S^2 $.

On a general Zoll manifold, such resemblance is not clear to the knowledge of the author. That is why we will restrict ourselves to Zoll surfaces of revolution.

\item Global geometry: Spectral distribution of the Laplacian-Beltrami operator. See Proposition~\ref{prop:spectral-zoll}. This works for Zoll manifolds of arbitrary dimension, and states that the Laplacian spectrum on Zoll manifolds of dimension~$ d $ is similar to that of the spherical Laplacian on~$ \S^d $.
\end{itemize}
It is worth comparing to the work of Burq-Gérard~\cite{B-G-Stabilization-Wave-Tori} of a similar stabilization problem on tori, where only the local geometry is consulted. 

\begin{prop}[Duistermaat-Guillemin~\cite{D-G-Zoll}]
\label{prop:spectral-zoll}
Let $ \Delta $ be the Laplacian-Beltrami operator on a  Zoll manifold of dimension~$ d $, then  
\begin{equation*}
\spec(-\Delta)\subset \bigsqcup_{n\ge 0} I_{n},
\end{equation*}
where~$ \{I_n\}_{n\ge 0} $ is a family of mutually disjointed intervals, such that 
\begin{equation}
\label{eq:interval-spectrum-zoll} 
I_{n} \subset {}\big](n+\beta/4)^{2}-A,(n+\beta/4)^{2}+A\big[{}
\end{equation}
for some $ \beta > 0 $, $ A > 0 $.
\end{prop}
\begin{rem}
\label{remark:spectral-zoll-dim-2}
When~$ d=2 $, we have~$ \beta = 2 $. See Proposition~4.35 of~\cite{Besse}.
\end{rem}
\begin{ex}
In particular, let $\Delta_d$ denote the spherical Laplacian on $ \S^d $, then (see Lemma~\ref{LEM::Spherical-Harmonics})
\begin{equation*}
\spec(-\Delta_{d})=\big\{(n+\tfrac{d-1}{2})^{2}-\tfrac{(d-1)^{2}}{4}:n\in\N\big\}.
\end{equation*}
We simple let $\beta/4=(d-1)/2$, and let $A$ be strictly larger than $(d-1)^{2}/4$.
\end{ex}

\begin{figure}[htb]
\label{fig:zoll-surface-of-revolution}
\caption{Zoll surface of revolution}
\begin{tikzpicture}
\draw [<-,dashed] (0,5) -- (0,4.5) node [above left] {$ N $} -- (0,0.5) node [below right] {$ S $} -- (0,0);
\draw [fill] (0,4.5) circle [radius=1.5pt] (0,0.5) circle [radius=1.5pt] (2,1.5) circle [radius=1.5pt] (1,3.72) circle [radius=1.5pt];
\draw plot [smooth cycle, tension = 0.7] coordinates {(0,4.5) (2,1.5) (0,0.5) (-2,1.5)};
\draw [<->,dashed] (2,1.5) node [right] {$ (\ell_0,\varphi) $} -- (0,1.5);
\draw (-2,1.5) arc (180:360:2 and 0.4);
\draw [dashed] (2,1.5) arc (0:180:2 and 0.4);
\draw (1,1.75) node [above] {$ r(\ell_0) = 1 $} (1,3.72) node [above right] {$ (\ell,\varphi) $} (-1,1.2) node [below] {$ \Gamma $};
\end{tikzpicture}
\end{figure}
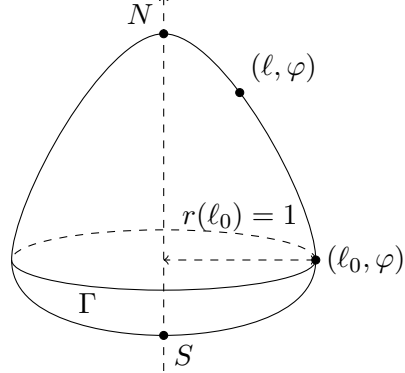

Let~$ \Sigma $ denote a Zoll surface of revolution, we state some local geometries of~$ \Sigma $. More details could be found in~\cite{Besse}. For an intuitive understanding, see Figure~\ref{fig:zoll-surface-of-revolution}. It is known that~$ \Sigma $ is automatically diffeomorphic to~$\S^{2}$, and there exists exactly two distinct points, respectively called the north pole and the south pole, denoted by~$N$ and~$S$, which are invariant under the actions of~$\S^{1}$. We then parametrize the surfaces by coordinates
\begin{equation*}
(\ell,\varphi) \in [0,\mathrm{dist}(N,S)] \times \S^1,
\end{equation*}
where $\ell$ is the arc-length parameter of one (and consequently every) geodesic from~$N$ to~$S$, and~$\varphi$ is the rotational angle corresponding to the actions of~$\S^{1}$, so that the Riemannian metric on~$ \Sigma $ is of the form 
\begin{equation*}
g=\d \ell^{2}+r(\ell)^{2}\d\varphi^{2},
\end{equation*}
where~$r(\ell)$ is the distance from the point~$(\ell,\varphi)$ to the axis of rotation. By Lemma~4.9 of~\cite{Besse}, there exists a unique~$\ell_{0}$ such that~$r(\ell)$ attains its maximum at~$\ell=\ell_{0}$. There is no loss of generality by assuming that $ r(\ell_0) = 1 $. Moreover we have $r'(\ell_{0})=0$, $r''(\ell_{0})<0$. 
The curve $\Gamma=\{(\ell_0,\varphi) : \varphi \in \S^1\}$ defines a closed geodesic of period~$ 2\pi $ (because $ r(\ell_0) = 1 $) called the \textit{equator}, while the regions $\Sigma^{+}=\{(\ell,\varphi) : \ell>\ell_{0}\}$ and $\Sigma^{-}=\{(\ell,\varphi) : \ell<\ell_{0}\}$ are called the \textit{upper and lower hemi-surfaces} respectively. Similarly to~$ \S^2 $, all geodesics on~$ \Sigma $ enter~$ \Sigma^+ $ except for the equator~$ \Gamma $.
\begin{rem}
\label{remark:local-geometry-equator}
If we denote $ c = - r''(\ell_0) / 2 > 0 $, then
\begin{equation*}
r(\ell) = 1 - c (\ell -\ell_0)^2 + O((\ell -\ell_0)^3).
\end{equation*}
This local geometry will be essential in performing a microlocal analysis near~$ \Gamma $ that proves our main theorem (Theorem~\ref{thm:Stability-Zoll-non-GCC}). In particular, if $ \Sigma = \S^2 $, then we take $ r(\ell) = \cos \ell $, such that $ (\ell,\varphi) \in [-\pi/2,\pi/2] \times \S^1 $ parametrizes~$ \S^2 $. In this case $ \ell_0 = 0 $, and
\begin{equation*}
r(\ell) = \cos \ell = 1 - \frac{1}{2} \ell^2 + O(\ell^3).
\end{equation*}
\end{rem}

\begin{rem}
\label{remark:criteria-zoll}
Using the change of variable $r(\ell)=\sin\theta$, to describe a Zoll surface of revolution, it is equivalent to give a Riemannian metric to~$ \S^2 $. By an abuse of notation, we still use~$ g $ to denote the metric on $ \S^d $ obtained by this isometry. If we parametrize~$ \S^2 $ by~$(\theta,\varphi)$, where~$\theta$ is the latitude while~$\varphi$ is the longitude of~$\S^{2}$, then by Corollary~4.16 of~\cite{Besse}, $(\S^{2},g)$ is a Zoll surface of revolution if and only if
\begin{equation*}
g=(1+h(\cos\theta))^{2}\d\theta^{2}+\sin^{2}\theta\d\varphi^{2},
\end{equation*}
for some smooth odd function $h$ from $[-1,1]$ to $(-1,1)$ with $h(1)=h(-1)=0$.
\end{rem}

Now we state the main result of this paper.

\begin{thm}
\label{thm:Stability-Zoll-non-GCC}
Let $\Sigma$ be a Zoll surface of revolution, then $a=1_{\Sigma^{+}}$ strongly stabilizes~\eqref{eq:Damped-Wave-Equation}.
\end{thm}
\begin{rem}
As a direct consequence of our proof, in order for~$ a $ to strongly stabilize~\eqref{eq:Damped-Wave-Equation}, it suffices for~$ a $ to be bounded from below by a positive constant in a half-neighborhood of the equator. To be precise, this means that there exists some $ \varepsilon > 0 $, $ \delta > 0 $, such that 
\begin{equation*} 
a(\ell,\varphi) \ge \delta \cdot 1_{\ell_0 < \ell < \ell_0+\varepsilon}(\ell,\varphi).
\end{equation*}
However, we will only prove the case when $ a=1_{\Sigma^{+}} $ for simplicity.
\end{rem}

\subsection{Stabilization of Damped Waves on $\S^{d}$}

\label{sec:Stabilization-S^d-Lebeau}

In this section we prove Theorem~\ref{thm:Lebeau}. First we recall the following classical result, due to J.-L.~Lions~\cite{Lions}.
\begin{prop}
\label{prop:Stability=Observability}
Let $ (M,g) $ be a compact Riemannian manifold without boundary, and let $ 0 \le a \in L^\infty(M) $, then the following two statements are equivalent.
\begin{enumerate}
\item $ a $ strongly stabilizes~\eqref{eq:Damped-Wave-Equation}.
\item For some $T>0$, $C>0$, and for every solution~$ u $ to the Cauchy problem of the undamped wave equation
\begin{equation}
\label{eq:undamped-wave-eq}
\begin{cases}
(\partial_{t}^{2}-\Delta)u=0, & \mathrm{in\ }\mathcal{D}'(\R\times M);\\
(u,\pt u)_{t=0}=(u_{0},u_{1}), & \in H^{1}(M)\times L^{2}(M),
\end{cases}
\end{equation}
the following observability inequality holds,
\begin{equation}
\label{eq:ineq-observability}
E(u,0)\le C\int_{0}^{T}\int_{M}a|\pt u|^{2}\dx\,\dt.
\end{equation}
\end{enumerate}
\end{prop}

Therefore it remains to establish this observability inequality. Coming back to~$ \S^d $, we recall some basic properties of the spherical Laplacian and spherical harmonics (see for example Chapter~IV, Section~2 of~\cite{S-W}).
\begin{lem}
\label{LEM::Spherical-Harmonics} Let $\Delta_{d}$ denote the
spherical Laplacian on $\S^{d}$, then 
\begin{enumerate}
\item $\spec(-\Delta_{d})=\big\{\lambda_{n}^{2}=n(n+d-1)=(n+\frac{d-1}{2})^{2}-\frac{(d-1)^{2}}{4}:n\in\N\big\}$. 
\item The eigenspace~$E_{n}$ to~$ -\Delta_d $ of eigenvalue~$ \lambda^2_n $ consists of spherical harmonics of degree $n$, which are restrictions to $\S^{d}$ of harmonic polynomials of $d+1$ variables, homogeneous of degree $n$. In particular, if $u\in E_{n}$, then $u(-x)=(-1)^{n}u(x)$. 
\end{enumerate}
\end{lem}
As a consequence, each $ u \in H^s(\S^d) $ with $ s \in \R $ admits a unique decomposition in distributional sense of the following form,
\begin{equation*}
u = \sum_{n \ge 0} u_n, \quad \mathrm{with\ } u_n \in E_n.
\end{equation*}
This allows us to specify the~$ H^s(\S^d) $ norm in terms of this decomposition by setting
\begin{equation*}
\|u\|_{H^s}^2 = \|(1 + \Delta)^{s/2} u\|_{L^2}^2 = \sum_{n \ge 0} \langle \lambda_n \rangle^{2s} \|u_n\|_{L^2}^2, \quad \mathrm{with\ } \langle \lambda_n \rangle = \sqrt{1 + \lambda_n^2}.
\end{equation*}
We then introduce a new differential operator as a perturbation of~$ -\Delta_d $,
\begin{equation}
\label{eq:def-L-S^d}
\L=-\Delta_{d}+\tfrac{(d-1)^{2}}{4}.
\end{equation}
The advantage of~$ \L $ to~$ -\Delta_d $ is that the spectrum of~$ \L $ consists of \textit{exact} squares of arithmetic sequence, $ \spec(\L) = \{(n+\frac{d-1}{2})^{2} : n\in\N\} $, so that
\begin{equation}
\label{eq:Spectral-Perturbated-Spherical-Laplacian}
\spec(\sqrt{\L})=\big\{n+\tfrac{d-1}{2}:n\in\N\big\}.
\end{equation}
Solving the following Cauchy problem
\begin{equation}
\label{eq:equation-L-S^d}
\begin{cases}
(\partial_{t}^{2}+\L)u=0, & \mathrm{in\ }\mathcal{D}'(\R\times\S^{d});\\
(u,\pt u)_{t=0}=\left(u_{0},u_{1}\right), & \in H^{1}(\S^{d})\times L^2(\S^{d}),
\end{cases}
\end{equation}
by using Fourier series,
\begin{equation}
\label{eq:solution-L-Fourier}
\begin{split}
u(t) & = \cos(t\sqrt{\L}) u_0 + \sqrt{\L}^{-1} \sin(t\sqrt{\L}) u_1 \\
& = \begin{cases}
\sum_{n\ge 0}\big(e^{it(n+\frac{d-1}{2})}u_{n}^{+}+e^{-it(n+\frac{d-1}{2})}u_{n}^{-}\big), & d\ge 2, \\
u_{0}^{0}+u_{0}^{1}t+\sum_{n\ge 1}\big(e^{itn}u_{n}^{+}+e^{-itn}u_{n}^{-}\big), & d=1,
\end{cases}
\end{split}
\end{equation}
where we write $u_{0}=\sum_{n\ge 0}u_{n}^{0},\ u_{1}=\sum_{n\ge 0}u_{n}^{1}$, with $u^{i}_{n}\in E_{n}$, and by an explicit calculation, we have for $n\ge 0$ when $d\ge 2$ and $n\ge 1$ when $d=1$,
\begin{equation*}
u_{n}^{+}+u_{n}^{-}=u_{n}^{0},\quad i(n+\tfrac{d-1}{2})(u_{n}^{+}-u_{n}^{-})=u_{n}^{1}.
\end{equation*}
If we assume $ (u_0,u_1) \in H^s \times H^{s-1} $ for some $ s \in \R $, then this expression gives an a priori bound for $ \|u\|_{L^\infty_\loc H^s} $. Indeed, for $ d \ge 2 $ ($ d = 1 $ is similar), by the characterization of the~$ H^s $ norm, and the triangular inequality,
\begin{equation}
\label{eq:estimate-a-priori-wave}
\begin{split}
\|u(t)\|_{H^s}^2 
& = \sum_{n \ge 0} \langle \lambda_n \rangle^{2s} \|e^{it(n+\frac{d-1}{2})}u_{n}^{+}+e^{-it(n+\frac{d-1}{2})}u_{n}^{-}\|_{L^2}^2 \\
& = \sum_{n \ge 0} \langle \lambda_n \rangle^{2s} \|\cos\big(t(n+\tfrac{d-1}{2})\big) u^0_n + (n+\tfrac{d-1}{2})^{-1} \sin\big(t(n+\tfrac{d-1}{2})\big) u^1_n\|_{L^2}^2 \\
& \lesssim \sum_{n \ge 0} \langle \lambda_n \rangle^{2s} \|u^0_n\|_{L^2}^2 + \sum_{n \ge 0} \langle \lambda_n \rangle^{2s} (n+\tfrac{d-1}{2})^{-2} \|u^1_n\|_{L^2}^2 \\
& \lesssim \|u_0\|_{H^s}^2 + \|u_1\|_{H^{s-1}}^2.
\end{split}
\end{equation}
When $ d \ge 2 $, we obtain $ \|u\|_{L^\infty H^s} \lesssim \|u_0\|_{H^s} + \|u_1\|_{H^{s-1}} $, while for $ d = 1 $, the same estimate holds after replacing $ \|u\|_{L^\infty H^s} $ with $ \|u\|_{L^\infty_\loc H^s} $, due to the linear growth in time of the term~$ u_0^1 t $.

Observe that in the expression of the solution, the family of factors $\{e^{\pm it(n+\frac{d-1}{2})}\}_{n\in\N}$ are orthogonal in $L^2([0,2\pi])$. This fact makes the observability of~\eqref{eq:undamped-wave-eq} easier to prove, due to the following two reduction lemmas.

\begin{defn}
We say that $ a $ observes~\eqref{eq:equation-L-S^d} if for some constants $ T > 0 $, $ C > 0 $ and every solution~$ u $ to~\eqref{eq:undamped-wave-eq}, the observability inequality~\eqref{eq:ineq-observability} holds. 
We say that $ a $ observes the spherical harmonics, if for some $C>0$, and every spherical harmonic $v \in \cup_{n\in\N} E_{n}$,
\begin{equation}
\label{eq:ineq-observability-spherical-harmonics}
\|a^{1/2}v\|_{L^2(\S^{d})}\ge C \|v\|_{L^2(\S^{d})}
\end{equation}
\end{defn}

\begin{lem}
\label{lem:observability-reduction-to-perturbed-wave-S^d}
For $ M = \S^d $, let $ 0 \le a \in L^\infty(\S^d) $, if~$ a $ observes~\eqref{eq:equation-L-S^d}, then~$ a $ observes~\eqref{eq:undamped-wave-eq}.
\end{lem}
\begin{proof}
Let $u$ solve~\eqref{eq:undamped-wave-eq}. We decompose $u=v+w$ such that
\begin{equation*}
\begin{cases}
(\pt^{2}+\L)v=0, & (v,\pt v)_{t=0}=(u_{0},u_{1});\\
(\pt^{2}+\L)w=\frac{(d-1)^{2}}{4} u, & (w,\pt w)_{t=0}=(0,0).
\end{cases}
\end{equation*}
Now that $ a $ observes~\eqref{eq:equation-L-S^d}, for some $ T > 0 $,
\begin{align*}
E(u,0)=E(v,0) & \lesssim \int_{0}^{T}\int_{\S^{d}}a|\pt v|^{2}\dx\,\dt 
\lesssim\int_{0}^{T}\int_{\S^{d}}a|\pt u|^{2}\dx\,\dt 
+ \int_{0}^{T}\int_{\S^{d}}a|\pt w|^{2}\dx\,\dt.
\end{align*}
By Duhamel's formula, $ \pt w(t) = \frac{(d-1)^{2}}{4} \int_{0}^{t}\cos\big((t-s)\sqrt{\L}\big)u(s)\,\ds$. Then we use the boundedness $ \| \cos(t\sqrt{\L}) \|_{L^2\to L^2} \le 1 $, and the a priori estimate~\eqref{eq:estimate-a-priori-wave}, to obtain
\begin{equation*}
\|\pt w(t)\|_{L^{2}}^{2} 
\lesssim\int_{0}^{T}\|u(s)\|_{L^{2}}^{2}\ds
\lesssim \|(u_{0},u_{1})\|_{L^{2}\times H^{-1}}^{2}.
\end{equation*}
Combine the inequalities above, we obtain a weak observability,
\begin{equation*}
E(u,0)\lesssim \int_{0}^{T}\int_{\S^{d}}a|\pt u|^{2}\dx\,\dt + \|(u_{0},u_{1})\|_{L^{2}\times H^{-1}}^{2}.
\end{equation*}
Then it is a classical argument of uniqueness-compactness due to Bardos-Lebeau-Rauch~\cite{B-L-R} which allows us to remove the compact remainder term $\|(u_{0},u_{1})\|_{L^{2}\times H^{-1}}^{2}$ and obtain the (strong) observability. This amounts to prove by contradiction and extract a subsequence of solutions of~\eqref{eq:undamped-wave-eq} which violates the observability, but converges strongly in the energy norm due to the compactness given by the weak observability. This gives us a solution to~\eqref{eq:undamped-wave-eq} with non vanishing energy (the energy is now conserved in time because there is no damping term in~\eqref{eq:undamped-wave-eq}), say $ v $, such that $ a \pt v = 0 $. Then we conclude by showing that, for~$ a \not\equiv 0 $, such solution does not exist (the only solution to~\eqref{eq:undamped-wave-eq} with $ a \pt v = 0 $ must be constant, and hence with zero energy). For more details, see the proof of Lemma~\ref{LEM::Reduction-Step-2}.
\end{proof}

\begin{lem}
\label{lem:Observability-Spherical-Harmonics}
If~$ a $ observes the spherical harmonics, then~$ a $ observes~\eqref{eq:equation-L-S^d}.
\end{lem}
\begin{proof}
We only prove the lemma for~$ d \ge 2 $, the proof for~$ d = 1 $ is almost the same.

We set $T=2\pi$, and use Fubini's theorem, the explicit formula for solutions~\eqref{eq:solution-L-Fourier}, the orthogonality of the family $\{e^{\pm it(n+\frac{d-1}{2})}\}_{n\in\N}$ in $L^2([0,2\pi])$, the observability~\eqref{eq:ineq-observability-spherical-harmonics}, and the characterization of Sobolev norms by spherical harmonics,
\begin{align*}
& \phantom{{}={}} \int_{0}^{2\pi}\int_{\S^d} a|\pt u|^{2}\dx\,\dt \\
& = \int_{\S^{d}} a(x) \int_{0}^{2\pi} \Big| \sum_{n\ge 0}\big(n+\tfrac{d-1}{2}\big)\big(e^{it(n+\frac{d-1}{2})}u_{n}^{+}(x)-e^{-it(n+\frac{d-1}{2})}u_{n}^{-}(x)\big) \Big|^{2}\dt\,\dx  \\
& = \int_{\S^d}a(x)\sum_{n\ge 0}\big(n+\tfrac{d-1}{2}\big)^{2}\big(|u^{+}_{n}(x)|^{2}+|u^{-}_{n}(x)|^{2}\big)\,\dx  \\
& \gtrsim \int_{\S^d}\sum_{n\ge 0}\big(n+\tfrac{d-1}{2}\big)^{2}\big(|u^{+}_{n}(x)|^{2}+|u^{-}_{n}(x)|^{2}\big)\,\dx  \\
& \gtrsim \int_{\S^d}\sum_{n\ge 0}\big(n+\tfrac{d-1}{2}\big)^{2}\big(|u^{+}_{n}(x)+u^-(x)|^{2}+|u^+(x)-u^{-}_{n}(x)|^{2}\big)\,\dx  \\
& \gtrsim \int_{\S^d}\sum_{n\ge 0} \big(n+\tfrac{d-1}{2}\big)^{2} |u^0_n(x)|^2 
+ \int_{\S^d}\sum_{n\ge 0} |u^1_n(x)|^2 \dx \\
& \ge E(u,0).
\end{align*}
\end{proof}

Then we finish the proof of Theorem~\ref{thm:Lebeau} by showing that~$ a(x) = 1_{\S^d_+}(x) $ observes the spherical harmonics.
\begin{prop}
On $ \S^d $, $ a(x) = 1_{\S^d_+}(x) $ observes the spherical harmonics.
\end{prop}
\begin{proof}
This comes easily from the symmetry properties of spherical harmonics stated in Lemma~\ref{LEM::Spherical-Harmonics}. Indeed, if $v\in E_{n}$, then $ v(-x) = (-1)^n v(x) $ implies that 
\begin{equation*}
\|v\|_{L^2(\S^d_+)} = \|v\|_{L^2(S^d_-)},
\end{equation*}
whence the observability
\begin{equation*}
\|a^{1/2}v\|_{L^2(\S^d)}=\|v\|_{L^2(\S^d_{+})}=\tfrac{1}{\sqrt{2}}\|v\|_{L^2(\S^d)}.
\end{equation*}
\end{proof}

\subsection{Strategy of Proof}

\subsubsection{Proof of Theorem~\ref{thm:Lebeau}}

We first analyse the proof of Theorem~\ref{thm:Lebeau} presented above, which consists of the following 4~steps.

\paragraph{\textbf{Step 1}}
Reduce the strong stabilization of the damped wave equation~\eqref{eq:Damped-Wave-Equation} to the observability~\eqref{eq:ineq-observability} of the undamped wave equation~\eqref{eq:undamped-wave-eq}. This is a classical argument.

\paragraph{\textbf{Step 2}}
Reduce the observability of the undamped wave equation~\eqref{eq:undamped-wave-eq} to the observability of the perturbed wave equation~\eqref{eq:equation-L-S^d}. This perturbation uses essentially the fact that the spectrum of the spherical Laplacian is distributed near squares of an arithmetic sequence $ \{(n+\frac{d-1}{2})^2\}_{n\in\N} $. In fact, the spectrum is exactly of distance~$ \frac{(d-1)^2}{4} $ away from this sequence. Therefore, by adding to~$ -\Delta_d $ the constant~$ \frac{(d-1)^2}{4} $, we obtain an operator~$ \L $, the spectrum of whose square root is exactly the arithmetic sequence $ \{n+\frac{d-1}{2}\}_{n\in\N} $.

\paragraph{\textbf{Step 3}}
Reduce the observability of the perturbed wave equation~\eqref{eq:equation-L-S^d} to the observability of spherical harmonics, that is~\eqref{eq:ineq-observability-spherical-harmonics}.
To do so, we solve~\eqref{eq:equation-L-S^d} explicitly with Fourier series (that is, decomposition in spherical harmonics), and use the orthogonality of the time factors 
\begin{equation*}
\{e^{\pm it\lambda}\}_{\lambda \in \spec(\sqrt{\L})} = \{e^{\pm it(n+\frac{d-1}{2})}\}_{n\in\N}
\end{equation*}
in~$ L^2([0,2\pi]) $ to decouple the space and time variables. In this way, the time variable can be omitted, and we are left only to consider the spherical harmonics.

\paragraph{\textbf{Step 4}}
Prove the observability of spherical harmonics. We use the symmetry of spherical harmonics to show that the $ L^2 $~norm of a spherical harmonic is equally distributed on upper and lower hemispheres.

\subsubsection{Proof of Theorem~\ref{thm:Stability-Zoll-non-GCC}}
We will follow this strategy to prove Theorem~\ref{thm:Stability-Zoll-non-GCC}, but with the following modifications.

\paragraph{\textbf{Step 1}} 
Same as above.

\paragraph{\textbf{Step 2}}
The only (slight) difference is the definition of the perturbed wave equation, because the perturbation~$ \L $ of the Laplacian-Beltrami operator~$ -\Delta $ on a Zoll surface of revolution~$ \Sigma $ can not be so simply defined as~\eqref{eq:def-L-S^d}. To define~$ \L $ in this situation, we recall Proposition~\ref{prop:spectral-zoll} and Remark~\ref{remark:spectral-zoll-dim-2}. For~$ \lambda \ge 0 $ such that~$ \lambda^2 \in \spec(-\Delta) $, we let~$ E_\lambda $ denote the (minus) Laplacian eigenspace of eigenvalue~$ \lambda^2 $, and set for $ n \ge 0 $ the linear space
\begin{equation*}
\tilde{E}_n = \bigoplus_{\lambda^2 \in I_n} E_\lambda.
\end{equation*}
Then~$ \L $ is defined by prescribing its action on each~$ \tilde{E}_n $,
\begin{equation*}
\L|_{\tilde{E}_{n}}=(n+1/2)^2\,\Id_{\tilde{E}_{n}}.
\end{equation*}
Therefore $ \tilde{E}_n $ are eigenspaces of~$ \L $, whose elements will be called $ \L $-eigenfunctions, and
\begin{equation*}
\spec(\sqrt{\L}) \subset \{ n + 1/2 : n \in \N \}.
\end{equation*}
Moreover, by~\eqref{eq:interval-spectrum-zoll}, if we set~$ K = \Delta + \L $, then $ \|K\|_{\tilde{E}_n\to\tilde{E}_n} \le A $, where $ \tilde{E}_n $ is equipped with the $ L^2(\Sigma) $ norm. Consequently, by the orthogonal direct sum decomposition $ L^2(\Sigma) = \oplus_{n \ge 0} \tilde{E}_n $, we show that~$ K $ is a bounded operator on $ L^2(\Sigma) $,
\begin{equation}
\label{eq:K-norm}
\|K\|_{L^2(\Sigma) \to L^2(\Sigma)} \le A,
\end{equation}
which plays the same role as the constant~$ \frac{(d-1)^2}{4} $ in the spherical case. Then the same argument shows that the observability for~\eqref{eq:undamped-wave-eq} can be deduced from the observability of the following perturbed wave equation,
\begin{equation}
\label{eq:equation-L-zoll}
\begin{cases}
(\partial_{t}^{2}+\L)u=0, & \mathrm{in\ }\mathcal{D}'(\R\times\Sigma);\\
(u,\pt u)_{t=0}=\left(u_{0},u_{1}\right), & \in H^{1}(\Sigma)\times L^2(\Sigma).
\end{cases}
\end{equation}

\paragraph{\textbf{Step 3}}
Reduce the observability of~\eqref{eq:equation-L-zoll} to the observability of $ \L $-eigenfunctions, that is to say, for some~$ C>0 $, and every $ u \in \cup_{n\ge 0} \tilde{E}_n $,
\begin{equation}
\label{eq:observability-L-eigenfunction}
\|a^{1/2} u\|_{L^2(\Sigma)} \ge C \|u\|_{L^2(\Sigma)}.
\end{equation}
Recall that in our case, $ a(x) = 1_{\Sigma^+}(x) $. To do this, we use the orthogonality of the time factors $ \{e^{\pm it(n+1/2)}\}_{n\in\N} $ in $ L^2([0,2\pi]) $, which comes with luck from the fact that $ \beta = 2 $ on dimension~$ 2 $ (recall Remark~\ref{remark:spectral-zoll-dim-2}), so that $ n + \beta/4 = n + 1/2 $. However, this fact is not necessary, for we can always use Ingham's inequality (see the original work of Ingham~\cite{Ingham}, see also~\cite{Zuazua} for its application in the theory of control).

\paragraph{\textbf{Step 4}}
Prove the observability of $\L$-eigenfunctions~\eqref{eq:observability-L-eigenfunction}. Unfortunately, the simple proof for the observability of spherical harmonics does not apply, because neither the $\L$-eigenfunctions nor the Laplacian eigenfunctions on~$ \Sigma $ share such strong symmetries as the spherical harmonics. However, we observe that, by the definition of~$ \L $, the $ \L $-eigenfunctions are quasi-modes. Indeed, let $ u \in \tilde{E}_n $, normalized in $ L^2 $~norm, that is, $ \|u\|_{L^2(\Sigma)} = 1 $; introduce the semiclassical parameter $ h = (n + 1/2)^{-1} $, then by~\eqref{eq:K-norm}
\begin{equation*}
(-h^2 \Delta + 1) u = - h^2 K u = O(h^2)_{L^2}.
\end{equation*}
This suggests a proof by contradiction and analyzing the semiclassical defect measures (see Gérard~\cite{Gerard-Semiclassical-Measure}, Gérard-Leichtnam~\cite{G-L}, Lions-Paul~\cite{Lions-Wigner}, see also~\cite{Burq-Bourbaki}) of a sequence of $ \L $-eigenfunctions, which violates the observability, that is, $ \|1_{\Sigma^+}u\|_{L^2(\Sigma)} = o(1) $. Such argument is originally due to G.~Lebeau, dating back to his work~\cite{Lebeau} which uses the propagation of (classical) defect measures; see for example~\cite{Burq-1,Zworski} for the semiclassical setting.

A classical argument shows that such semiclassical defect measure, say~$ \mu $, is supported on the unit cotangent bundle $ S^*\Sigma $, vanishes on~$ T^*\Sigma^+ $, and is invariant by the (co-)geodesic flow. Therefore~$ \mu $ carries no mass on the union of geodesics which enter~$ \Sigma^+ $. Recall that on~$ \Sigma $, every geodesic enter~$ \Sigma^+ $ within the period of the geodesic flow (which is, in our case, $ 2\pi $, by the normalization $ r(\ell_0) = 1$), except for a rogue one, the equator~$ \Gamma $. We are thus unable to close the routine argument as the $ \L $-eigenfunctions may concentrate on~$ \Gamma $ (a simple example is $ \Sigma = \S^2 $, where the spherical harmonics $ u_n(x,y,z) = (x + iy)^n $ will concentrate on the equator~$ z = 0 $ as $ n \to \infty $); but to conclude that
\begin{equation*}
\supp \mu \subset S^*\Sigma \cap \{\ell = \ell_0, \xi = 0 \} = \{(\ell_0,\varphi,0,\pm 1) : \varphi \in \S^1 \}.
\end{equation*}

To deal with this problem, we take a closer look at the concentration behavior near the equator. It suffices to show that the speed of concentration from each side of the equator is comparable, so that the $ L^2 $ norm of this sequence of $ \L $-eigenfunctions must be comparably distributed on each side as well, which contradicts to our hypothesis that the observability from the upper hemi-surface is violated by this sequence. Such idea is achieved by some proper scalings of the latitude coordinate~$ \ell $, and is closely related to the \textit{second microlocalization} along the equator, as illustrated by~\cite{B-G-Stabilization-Wave-Tori}. It is explicitly carried out as follows:

\begin{enumerate}[nosep,leftmargin=*]
\item First, to simplify some calculations, we will work on an isothermal coordinate on~$ \Sigma $. There exists an strictly increasing $ f \in C^\infty(\R) $ such that
\begin{equation*}
f'(x) = r(f(x)), \quad f(0) = \ell_0.
\end{equation*}
Then under the change of variable $ \ell = f(x) $, the north pole~$ N $, the south pole~$ S $ and the equator~$ \Gamma $ now respectively corresponds to~$ x = -\infty $, $ x = \infty $ and $ x = 0 $. Denoting for simplicity $ \rho = f' $, the metric~$ g $ now writes under the coordinates~$ (x,\varphi) $ as
\begin{equation*}
g = \rho(x) (\dx^2 + \d\varphi^2) = \big(1 - cx^2 + O(x^3)\big) (\dx^2 + \d\varphi^2),
\end{equation*}
where the positive constant~$ c $ is the same as in Remark~\ref{remark:local-geometry-equator}; and the Laplacian-Beltrami operator writes
\begin{equation}
\label{eq:Laplacian-x-varphi}
\Delta = \rho(x)^{-1} (\partial_x^2 + \partial_\varphi^2) = \big( 1 + cx^2 + O(x^3) \big) (\partial_x^2 + \partial_\varphi^2).
\end{equation}
We also remark that under these coordinates, 
\begin{equation*}
L^2(M) \simeq L^2(\rho^2\dx,\R) \otimes L^2(\d\varphi,\S^1).
\end{equation*}

\item On a general compact surface of revolution, $ -\Delta $ is invariant under rotation, and commutes with the infinitesimal generator of rotation, that is, $ D_\varphi = \frac{1}{i} \partial_\varphi $. We expect each Laplacian eigenspace to be a direct sum of $ D_\varphi $-eigenspaces. Indeed, on~$ \Sigma $, for $ \lambda^2 \in \spec(-\Delta) $, the following decomposition holds,
\begin{equation*}
E_\lambda = \bigoplus_{k \in \Z} e^{ik\varphi} A_{\lambda,k},
\end{equation*}
where $ A_{\lambda,k} $ consists of smooth functions of variable~$ x $, such that, whenever $ w \in A_{\lambda,k} $, we have $ w \in L^2(\rho^2 \dx,\R) $; and $ u(\varphi,x) = e^{ik\varphi} w(x) \in L^2(M) $ is a common eigenfunction of~$ -\Delta $ and~$ D_\varphi $,
\begin{equation*}
-\Delta u = \lambda^2 u, \quad D_\varphi u = k u.
\end{equation*}
By~\eqref{eq:Laplacian-x-varphi}, we have a second order differential equation for~$ w $,
\begin{equation}
\label{eq:equation-w-intro}
-\partial_x^2 w + k^2 w = \lambda^2\rho^2 w = \lambda^2 (1 - cx^2 + O(x^3)) w.
\end{equation}
It is known that the Laplacian eigenfunctions are smooth, in particular at the poles~$ N $ and~$ S $. This gives a boundary condition for~$ w $, 
\begin{equation*}
\lim_{|x| \to \infty} \partial_x^n w(x) = 0, \quad \mathrm{when} \quad  k \ne 0, n \in \N.
\end{equation*}
Consequently, up to a multiplicative constant, there exists at most one solution to~\eqref{eq:equation-w-intro}, which means
\begin{equation*}
\dim A_{\lambda,k} \le 1, \quad \mathrm{if\ } k \ne 0.
\end{equation*}
The case $ k = 0 $ poses no problem because as we have seen, $ \supp \mu \subset \{\theta = 1\} $, therefore the terms with $ 1 - h^2 k^2 \to 0 $ (therefore $ k \sim h^{-1} \to \infty $) contribute to almost all of the total mass. Now we set
\begin{equation*}
\tilde{A}_{n,k} = \bigoplus_{\lambda^2 \in I_n} A_{\lambda,k},
\end{equation*}
and obtain the decomposition for $ \L $-eigenspaces,
\begin{equation*}
\tilde{E}_n = \bigoplus_{k\in\Z} e^{ik\varphi} \tilde{A}_{n,k}.
\end{equation*}

\item Due to the orthogonality of the family $ \{e^{ik\varphi}\}_{k \in \Z} $ in $ L^2(\d\varphi,\S^1) $, we are left to prove the following observability, that for any sequence $ \{\tilde{w}_{n,k} \in \tilde{A}_{n,k} \}_{n \in \N,k \in \Z}  $, where the indexes appearing in the sequence satisfy $ 1 - h^2 k^2 = o(1) $ as $ n \to \infty $ (recall that $ h = (n+1/2)^{-1} $; such a sequence will be called \textit{admissible}, see Definition~\ref{def:admissible}), there exits some $ C > 0 $, such that for any $ \tilde{w}_{n,k} \in \tilde{A}_{n,k} $ in the sequence, we have
\begin{equation}
\label{eq:observability-w_n_k}
\|1_{x > 0} \tilde{w}_{n,k}\|_{L^2(\rho^2\dx)} \ge C \|\tilde{w}_{n,k}\|_{L^2(\rho^2\dx)}.
\end{equation}
The weight $ \rho^2 $ is of no importance as $ \tilde{w}_{n,k} $ concentrates on $ x = 0 $ (For a rigorous argument, we will use an Lithner-Agmon type estimate). In order to prove~\eqref{eq:observability-w_n_k}, we observe that~$ \tilde{w}_{n,k} $ satisfies a 1-dimensional stationary semiclassical Schr\"{o}dinger equation,
\begin{equation*}
(-h^2 \partial_x^2 + V) \tilde{w}_{n,k} = E \tilde{w}_{n,k} + O(h^2)_{L^2\to L^2} \tilde{w}_{n,k},
\end{equation*}
where $ V = 1 - \rho^2 = c x^2 + O(x^3) $ near $ x = 0 $, and $ E = 1 - h^2 k^2 $. Then we argue by contradiction and extract a sequence $ k = k(n) $ and set $ \tilde{w}_n = \tilde{w}_{n,k} $ which violates the observability, and treat separately two cases, $ E = O(h) $ and $ E \gg h $ (we can show that $ E \gtrsim -h^2 $).
\begin{enumerate}[nosep,leftmargin=*]
\item If $ E = O(h) $, then we use the scaling $ z = c^{1/4} h^{-1/2} x $ to obtain a classical Schr\"{o}dinger equation,
\begin{equation*}
(-\partial_z^2 + z^2 + O(h^{1/2})) \tilde{w}_{n} = (F + o(1)) \tilde{w}_{n},
\end{equation*}
for some $ 0 \le F \in \R $, and shows that $ \tilde{w}_{n} $ is close to an eigenfunction of the harmonic oscillator $ -\partial_z^2 + z^2 $, which is either an even functions or an odd function, whose mass are thus equally distributed on each side of the origin $ z=0 $.

\item If $ E \gg h $, then we use another scaling $ z = c^{1/2} E^{-1/2} $, $ \h = c^{1/2} E^{-1} h $, and obtain a semiclassical Schr\"{o}dinger equation, with a semiclassical parameter~$ \h = o(1) $,
\begin{equation*}
(-\h \partial_z^2 + z^2 + o(1)) \tilde{w}_{n} = \tilde{w}_{n} + o(\h)_{L^2\to L^2} \tilde{w}_{n}.
\end{equation*}
The ($ \h $-)semiclassical measure of $ \tilde{w}_{n} $ will be supported on the circle
\begin{equation*}
\{ (z,\zeta) \in T^*\R_z : z^2 + \zeta^2 = 1\},
\end{equation*}
and is invariant by rotation (which is induced by the Hamiltonian flow generated by the principal symbol $ z^2 + \zeta^2 $). So the mass of~$ \tilde{w}_n $ are also asymptotically equally distributed on each side of the origin~$ z=0 $.
\end{enumerate}
\end{enumerate}

\subsection*{Acknowledgement} This work which started as as a Mémoire of Master 2 of Université Paris-Sud, is finished under the guidance of Nicolas Burq, to whom the author owes great gratitude. The author would also express gratitude to Gilles Lebeau, who gave originally the proof of the spherical cases, based on which is this work only possible. Finally the author would like to thank the referees whose careful reading and useful comments lead to a significant improvement in the presentation of the paper.

\section{Proof of Theorem~\ref{thm:Stability-Zoll-non-GCC}}

\subsection{Geometry of Zoll Surfaces of Revolution}
\label{sub:Geometry-of-Zoll's-Surfaces}

Let~$ \Sigma $ be a Zoll surface of revolution, we recall some of its basic geometric properties, referring to the monograph of Besse~\cite{Besse}.

\subsubsection{Coordinates and Geodesics}

$ \Sigma $ is diffeomorphic to $\S^{2}$, and admits a parametrization by local coordinates described as follows. Recall that~$\S^1$ acts smoothly, faithfully, and isometrically on~$ \Sigma $, leaving exactly two points fixed, which are called the \textit{north pole} and the \textit{south pole}, denoted respectively by~$N$ and~$S$. Fix a geodesic~$\gamma_{0}$ from~$N$ to~$S$, parametrized by arc length. We assume that the total length of~$ \gamma_0 $ is equal to~$ \pi $, after a proper normalization. Then as~$\varphi $ varies in~$\S^{1}$, $\gamma_{\varphi}=\varphi\gamma_{0}$ varies among all geodesics joining~$N$ and~$S$, which are called the \textit{meridians}. The coordinates on $U=\Sigma\backslash\{N,S\}$ is given by 
\begin{equation*}
U \ni \gamma_{\varphi}(\ell) \mapsto (\ell,\varphi) \in {}]0,\pi[{} \times \S^{1}.
\end{equation*}
The coordinate patches near~$N$ and~$S$ are respectively $U_{N}=\{N\}\cup\{(\ell,\varphi):0\le \ell<\pi\}$, and $U_{S}=\{S\}\cup\{(\ell,\varphi):0<\ell\le\pi\}$. They are diffeomorphic
to the 2-dimensional open ball $B(0,\pi)$ via the usual polar coordinates. The Riemannian metric on~$U$ has the form 
\begin{equation*}
g=\d \ell^{2}+r(\ell)^{2}\d\varphi^{2},
\end{equation*}
where $ r(\ell) $ is the distance from the axis of rotation (recall Figure~\ref{fig:zoll-surface-of-revolution}). Then~$ \Sigma $ being a Zoll surface of revolution means that the criteria stated in Remark~\ref{remark:criteria-zoll} is satisfied. 
There is a well defined differential operator $D_{\varphi}$ on $\Sigma$. It is the differential operator with respect to the the vector fields~$X$ on~$\Sigma$ defined as follows: To each point $m=(\ell,\varphi)\in\Sigma\backslash\{N,S\}$, we associate the unit tangent vector $Y(m)\in T_{m}\Sigma$, tangent to the \textit{parallel}~$\S^1 m$ (that is, the orbit of the point~$ m $ generated by the actions of~$ \S^1 $), with direction given by the positive orientation of~$\S^{1}$. Letting $X(m)=r(\ell)Y(m)$, $X(N)=0$, $X(S)=0$, then~$ X $ defines a smooth tangent vector field on~$ \Sigma $. For $u\in C^{\infty}(\Sigma)$, we define 
\begin{equation*}
D_{\varphi}u=\frac{1}{i}\langle\mathrm{d}u,X\rangle.
\end{equation*}
On $U$, we simply have $D_{\varphi}=\frac{1}{i}\partial_{\varphi}$, with $\partial_{\varphi}$ being the differentiation with respect to~$\varphi$. Therefore $D_{\varphi}$ is symmetric and commutes with $\Delta$, at least in a formal way,
\begin{equation*}
[\Delta,D_{\varphi}]=0.
\end{equation*}

Then we state a proposition concerning the geodesics of~$ \Sigma $.
\begin{prop} 
\label{prop:Geometry-Zoll-Surface-Revolution}
Let~$ \Sigma $ be a Zoll surface of revolution, and let~$ r $ be prescribed as above.
\begin{enumerate}
\item Then $r:[0,\pi]\to[0,1]$ is smooth, with $r(0)=r(\pi)=0,\ r'(0)=1,\ r'(\pi)=-1,\ r''(0)=r''(\pi)=0$. There exists a unique $\ell_{0}\in {}]0,\pi[{}$ such that $r(\ell_{0})=1$. Furthermore, $r'(\ell_{0})=0$, $r''(\ell_{0})<0$. The curve $\Gamma = \{(\ell_0,\varphi) : \varphi \in \S^1 \}$ is a geodesic called the equator. 
\item Apart from the equator, every geodesic is contained between a pair of parallels $\{\ell=\ell_1\}$ and $\{\ell=\ell_2\}$ for some $ \ell_1 < \ell_0 < \ell_2 $, and contacts each of
the parallel exactly once. 
\end{enumerate}
\end{prop}
\begin{cor}
From this proposition, every geodesic of~$\Sigma$ except for the equator~$\Gamma$ enters the upper hemi-surface $\Sigma^{+}=\{\ell>\ell_{0}\}$.
\end{cor}

To simplify later calculations, we will work on an isothermal coordinate defined on~$ U $ as follows. Let $ f \in C^\infty(\R) $ be the solution to the following first order ordinary differential equation,
\begin{equation*}
f'(x)=r(f(x)),\quad f(0)=\ell_0.
\end{equation*}
It is not difficult to see that (we refer to~\cite{Arnold}) 
\begin{equation*}
L<f<R,\quad  \lim_{x\to-\infty}f(x)=0,\quad \lim_{x\to\infty} f(x)=\pi.
\end{equation*}
Therefore $f$ defines a diffeomorphism $\R\simeq{}]0,\pi[{}$, with the equator now being $ \Gamma = \{ x = 0 \} $. Set $ x = f^{-1}(\ell) $, then the coordinates $(x,\varphi)$ are isothermal, indeed, 
\begin{equation*}
g=f'(x)^{2}\mathrm{d}x^{2}+r(f(x))^{2}\mathrm{d}\varphi^{2}
=\rho(x)^{2}(\mathrm{d}x^{2}+\mathrm{d}\varphi^{2}),
\end{equation*}
where $\rho(x):=r(f(x))=f'(x)$. We have $\rho \in {}]0,1]$, and $\rho(x)<1$ except for $x=0$, where $\rho(0)=1,\ \rho'(0)=0,\ \rho''(0)<0$. We also have, under these coordinates,
\begin{equation}
\label{eq:Ltwo-decomposition-Sigma}
L^2(\Sigma) = L^2(\rho^2 \dx,\R) \otimes L^2(\d\varphi,\S^1),
\end{equation}
and the Laplacian-Beltrami operator takes a simple form,
\begin{equation}
\label{Laplacian-Isothermal}
\Delta=\frac{1}{\rho(x)^{2}}(\partial_{x}^{2}+\partial_{\varphi}^{2}).
\end{equation}

\subsubsection{Laplacian Spectrum and Eigenfunctions}
\label{sub:Laplacian-Eigenfunctions-on-Surfaces-of-Revolution}

Recall that for some $ A > 0 $,
\begin{equation*}
\spec(-\Delta) \subset \bigsqcup_{n \ge 0} I_n, 
\quad \mathrm{with\ } I_{n} \subset {}](n+1/2)^{2}-A,(n+1/2)^{2}+A[{}.
\end{equation*}
For~$ \lambda \ge 0 $ such that~$ \lambda^2 \in \spec(-\Delta) $, we let~$ E_\lambda $ denote the (minus) Laplacian eigenspace of the eigenvalue~$ \lambda^2 $, and set for $ n \ge 0 $ the linear space
\begin{equation*}
\tilde{E}_n = \bigoplus_{\lambda^2 \in I_n} E_\lambda.
\end{equation*}
We define a linear (unbounded) operator~$ \L $ by a compact perturbation of~$ -\Delta $ such that
\begin{equation*}
\spec(\L) \subset \{(n+1/2)^2 : n \in \N \}.
\end{equation*}
Indeed, let $ \Pi_n : L^2(\Sigma) \to \tilde{E}_n $ denote the orthogonal projection, then we formally define
\begin{equation}
\label{eq:def-L}
\L = \sum_{n \ge 0} (n+1/2)^2 \Pi_n.
\end{equation}

Next we study the structure of~$ E_\lambda $. Since $-\Delta$ commutes with $D_{\varphi}$, it is natural to expect an orthogonal decomposition of~$ E_\lambda $ into $ D_\varphi $ eigenspaces. The following proposition is inspired by Beekmann~\cite{Beekmann}.
\begin{prop}
\label{prop:Laplacian-Eigenfunction-Evolution-Surface}
On each~$ E_\lambda $, we have a direct sum decomposition,
\begin{equation*}
E_\lambda = \bigoplus_{k \in \Z} e^{ik\varphi} A_{\lambda,k},
\end{equation*}
where $ A_{\lambda,k} \subset C^\infty(\R) $ is the solution space to
\begin{equation}
\label{1D-Laplacian-Regular}
-\partial_{x}^{2}w+k^{2}w=\lambda^{2}\rho^{2}w,
\end{equation}
with boundary conditions $ \lim_{|x| \to \infty} \partial_x^n w(x) = 0 $ for $ n \in \N $ and $ k \ne 0 $. In particular,
\begin{equation*}
\dim A_{\lambda,k} \le 1, \quad \mathrm{if} \quad k \ne 0.
\end{equation*}
If $ u(x,\varphi) = e^{ik\varphi} w(x) \in e^{ik\varphi} A_{\lambda,k} $, then
\begin{equation*}
-\Delta u = \lambda^2 u, \quad D_\varphi u = k u.
\end{equation*}
That is, $ e^{ik\varphi} A_{\lambda,k} $ are eigenspaces of~$ D_\varphi $, and the decomposition is thus orthogonal.
\end{prop}

\begin{proof}
The group action of~$\S^{1}$ on~$\Sigma$ induces naturally a group action on function spaces by $\varphi f=f\comp\varphi^{-1}$. Now that~$\S^{1}$ commutes with~$ -\Delta $, $E_{\lambda}$ is stable under~$\S^{1}$. It is known that the irreducible complex representations
of $\S^{1}$ are all one-dimensional of the form 
\begin{equation*}
\tau_k : \S^{1}\to U(1)\ \varphi\mapsto e^{ik\varphi},\quad k\in\Z.
\end{equation*}
Therefore~$ E_\lambda $ can be decomposed into $ \tau_k $-invariant subspaces, consisting of functions $ u(x,\varphi) $ satisfying
\begin{equation*}
u(x,\varphi)=\varphi^{-1}u(x,0)=e^{-ik\varphi}u(x,0),
\end{equation*}
which also shows that $ D_\varphi u = k u $. 

To obtain the equation satisfied by $ w \in A_{\lambda,k} $, it suffices to plug $ u(x,\varphi) = e^{ik\varphi} w(x) $ into the equation $ -\Delta u = \lambda^2 u $. The boundary condition for $ k \ne 0 $ comes evidently from the continuity of~$ D_\varphi^n u = k^n u $ at~$ N $ and~$ S $. To show that $ \dim A_{\lambda,k} \le 1 $, let~$w_{1}$ and~$w_{2}$ be two solutions to~\eqref{1D-Laplacian-Regular}, then their Wronskian $W(w_{1},w_{2})$, which is a constant by a direct calculation, vanishes at infinity by the boundary conditions. So these two solutions are linearly dependent.
\end{proof}

\begin{cor}
\label{cor:rotation-number-less-than-energy}
If $ 0 \ne |k| \ge \lambda $, then~$ A_{\lambda,k} = \{0\} $.
\end{cor}
\begin{proof}
Suppose $ w \in A_{\lambda,k} $ with $ 0 \ne |k| \ge \lambda $, then for $ n \in \N $,
\begin{equation*}
-\partial_{x}^{2}w+k^{2}w=\lambda^{2}\rho^{2}w,\quad \lim_{|x|\to\infty} \partial_x^n w(x)=0.
\end{equation*}
We will show that $w\in H^{1}(\R) \subset C(\R) $ (see Corollary \ref{cor:Solution-Approximate}), so it is legitimate to take~$L^{2}(\R)$ inner product between~$w$ and the equation to get
\begin{equation*}
0 \le \|\partial_x w\|_{L^{2}(\R)}^{2}
=(- \partial_x^2 w,w)_{L^{2}(\R)}
=\int_{\R}(\lambda^{2}\rho^{2}(x)-k^{2})|w(x)|^{2}\dx.
\end{equation*}
However, $0\le\rho\le1$, and that $\rho(x)<1$ except for $x=0$, we see that $\lambda^{2}\rho(x)^{2}-k^{2}<0$ except for $x=0$. Therefore $w(x)\equiv 0$ since it is continuous.
\end{proof}

\begin{cor}
\label{cor:orthogonality-A(lamnbda,k)}
For $ k \in \Z $, and $ \lambda_1 \ne \lambda_2 $,
\begin{equation*}
A_{\lambda_1,k} \perp A_{\lambda_2,k} \quad \mathrm{with\ respect\ to\ } L^2(\rho^2\dx,\R).
\end{equation*}
\end{cor}
\begin{proof}
For $ w_i \in A_{\lambda_i,k} $ with $ i=1,2 $, set $ u_i(x,\varphi) = e^{ik\varphi}w_i(x) $, then by~\eqref{eq:Ltwo-decomposition-Sigma},
\begin{equation*}
0 = (u_1,u_2)_{L^2(\Sigma)} = (e^{ik\varphi}w_1,e^{ik\varphi}w_2)_{L^2(\rho^2 \dx,\R) \otimes L^2(\d\varphi,\S^1)} =2\pi (w_1,w_2)_{L^2(\rho^2 \dx,\R)}.
\end{equation*}
\end{proof}

\begin{rem}
For $ n \in \N $, we set
\begin{equation*}
\tilde{A}_{n,k} = \bigoplus_{\lambda^2 \in I_n} A_{\lambda,k},
\end{equation*}
then it is an orthogonal direct sum with respect to $ L^2(\rho^2 \dx,\R) $. And we have
\begin{equation*}
\tilde{E}_n = \bigoplus_{k\in\Z} e^{ik\varphi} \tilde{A}_{n,k}.
\end{equation*}
\end{rem}

\subsection{Reduction to Observability of $ \L $-Eigenfunctions}

Since~$ \Sigma $ has no boundary, the energy of a solution~$ u $ to~\eqref{eq:Damped-Wave-Equation} does not control its zero frequency. In order to deal with this problem, we introduce the quotient Sobolev spaces
\begin{equation*}
\mathcal{H}^{s}(\Sigma)=H^{s}(\Sigma)/\C=\{[u]=u+\C:u\in H^{s}(\Sigma)\},
\end{equation*}
equipped with the quotient norms. We set in particular,
\begin{equation*}
\|{[u]}\|_{\mathcal{H}^{1}(\Sigma)} = \|\nabla u\|_{L^{2}(\Sigma)},
\end{equation*} 
so that, for $ u \in C(\R,H^1(\Sigma)) \cap C^1(\R,L^2(\Sigma)) $,
\begin{equation*}
E(u,t) = \frac{1}{2} \|(u(t),\pt u(t))\|_{\mathcal{H}^1 \times L^2}^2.
\end{equation*}
By the theorem of Hille-Yosida, we have
\begin{prop}
Define the quotient Laplacian by $[\Delta][u]=[\Delta u]$ for $[u]\in D([\Delta])=\{[u]\in\mathcal{H}^{1}(\Sigma):\Delta u\in L^{2}(\Sigma)\}$. Set
\begin{equation*}
[A]=\begin{pmatrix}0 & -[\Id]\\
-j\circ[\Delta] & a
\end{pmatrix}
\end{equation*}
with $D([A])=D([\Delta])\times H^{1}(\Sigma)$ where $[\Id]:H^{1}(\Sigma)\to\mathcal{H}^{1}(\Sigma)$ is the canonical projection, while $j:\mathcal{H}^{0}(\Sigma)\to L^{2}(\Sigma)$ associates each $[w]\in\mathcal{H}^{0}(\Sigma)$ a representative~$w$ such that $\int_{\Sigma}w\,\dx=0$. Then for all $([u_{0}],u_{1})\in\mathcal{H}^{1}(\Sigma)\times L^{2}(\Sigma)$, there exists a unique solution $([u],v)\in C(\R,\mathcal{H}^{1}(\Sigma))\times C^{1}(\R,L^{2}(\Sigma))$ of the equation 
\begin{equation}
\label{eq:equation-quotient-wave}
\begin{cases}
\pt\binom{[u]}{v}+[A]\binom{[u]}{v}=0,\\
([u],v)_{t=0}=([u_{0}],u_{1}).
\end{cases}
\end{equation}
Moreover, if~$u$ is the solution of~\eqref{eq:Damped-Wave-Equation} with initial data $(u_{0},u_{1})$, then $([u],\pt u)$ is the solution to~\eqref{eq:equation-quotient-wave}.
\end{prop}

\begin{prop}
\label{prop:Weak-Stability-Property-Criterion}
If $0 \le a \in L^\infty(\Sigma) $ and $ a \not\equiv 0 $, then~$ a $ weakly stabilizes~\eqref{eq:Damped-Wave-Equation}.
\end{prop}
\begin{proof}
The idea of the proof comes from~\cite{B-G}. By a density argument, it suffices to suppose that $(u_{0},u_{1})\in D(A)$, so that $([u_{0}],u_{1})\in D([A])$. Let~$u$ denote the corresponding solution to~\eqref{eq:Damped-Wave-Equation}. Observe that $E(u,t)=\frac{1}{2}\|([u],\pt u)\|_{\mathcal{H}}^{2}$ is non-increasing, and that~$ [A] $ commutes with the evolution of~\eqref{eq:equation-quotient-wave},
\begin{equation}
\label{eq:Bounded-in-D(A)}
\begin{split}
\Big\|\binom{[u]}{\pt u}\Big\|_{[A]} :=\Big\|\binom{[u]}{\pt u}\Big\|_{\mathcal{H}}+\Big\|[A]\binom{[u]}{\pt u}\Big\|_{\mathcal{H}}
\le\Big\|\binom{[u_{0}]}{u_{1}}\Big\|_{\mathcal{H}}+\Big\| A\binom{[u_{0}]}{u_{1}}\Big\|_{\mathcal{H}}<\infty.
\end{split}
\end{equation}
We claim that $(D([A]),\|\cdot\|_{[A]})\hookrightarrow\mathcal{H}$ is compact. Indeed, if $\binom{[u]_{n}}{v_{n}}$ is bounded in~$D([A])$, then $u_{n}$, $\Delta u_{n}$, $v_{n}$, $\nabla v_{n}$ are bounded in~$L^{2}(\Sigma)$. Up to a subsequence, $u_{n}-\int_{M}u_{n}\to u_{0}$ in~$H^{1}(\Sigma)$, so $[u_{n}]\to[u_{0}]$ in~$\mathcal{H}^{1}(\Sigma)$, and $v_{n}\to v_{0}$ in $L^{2}(\Sigma)$. By~\eqref{eq:Bounded-in-D(A)}, there exists a sequence $t_{k}\to+\infty$ such that $([u(t_{k})],\pt u(t_{k}))\rightharpoonup([v_{0}],v_1)$ weakly in~$D([A])$; and strongly in $\mathcal{H}^{1}(M)$ by compactness. Let~$v$ be the solution to~\eqref{eq:Damped-Wave-Equation} with initial data $(v_{0},v_{1})$, where~$v_{0}$ is the representative of~$[v_{0}]$ such that $\int_{\Sigma} v_{0}\dx = 0$, then 
\begin{equation*}
\begin{split}
E(v,t) 
& =\Big\|\binom{[v(t)]}{\pt v}\Big\|_{\mathcal{H}}
  =\Big\| e^{-t[A]}\binom{[v_{0}]}{v_{1}}\Big\|_{\mathcal{H}} 
  =\Big\| e^{-t[A]}\lim_{k\to\infty}e^{-t_{k}[A]}\binom{[u_{0}]}{u_{1}}\Big\|_{\mathcal{H}}\\
& =\Big\|\lim_{k\to\infty}e^{-(t+t_{k})[A]}\binom{[u_{0}]}{u_{1}}\Big\|_{\mathcal{H}}
  =\Big\|\binom{[v_{0}]}{v_{1}}\Big\|_{\mathcal{H}}=E(v,0).
\end{split}
\end{equation*}
So $v$ satisfies the undamped wave equation~\eqref{eq:undamped-wave-eq} as well.

We decompose the initial data as $ v_{0}=\sum_{\lambda}v_{\lambda}^{0}$ ,$ v_{1}=\sum_{\lambda} v_{\lambda}^{1} $, where $ \lambda $ varies in $ \spec(\sqrt{-\Delta}) $ and $ v^i_\lambda \in E_{\lambda} $. Then $ v^0_0 = 0 $, and
\begin{equation*}
v(t) = \cos(t\sqrt{-\Delta}) v_0 + \sqrt{-\Delta}^{-1}\sin(t\sqrt{-\Delta}) v_1
= v^1_0t+\sum_{\lambda \ne 0}\big(e^{it\lambda}v_\lambda^{+}+e^{-it\lambda}v_\lambda^{-}\big),
\end{equation*}
where for $ \lambda \ne 0 $,
\begin{equation*}
v_\lambda^{+} + v_\lambda^{-}=v_\lambda^{0},
\quad
i\lambda (v_\lambda^{+}-v_\lambda^{-})=v_\lambda^{1}.
\end{equation*}
Now fix $\lambda'\ne 0$, and set $w_{\lambda'}(T,x)=\frac{1}{T}\int_{0}^{T}\pt v(t,x)e^{-it\lambda'}\dt$. The fact that $a\pt v=0$ implies $aw_{\lambda'}=0$. An explicit calculation shows 
\begin{equation*}
\begin{split}
w_{\lambda'}(T)
=i\lambda'v_{\lambda'}^{+} +\sum_{\lambda \ne \lambda'} \frac{i\lambda}{iT(\lambda-\lambda')} \big(e^{iT(\lambda-\lambda')}-1\big) v_\lambda^{+} 
-\sum_{\lambda} \frac{i\lambda}{iT(\lambda+\lambda')} \big(e^{-iT(\lambda+\lambda')}-1\big)v_\lambda^{-}.
\end{split}
\end{equation*}
This implies that, as $T\to\infty$, $w_{\lambda'}(T)\to i\lambda'v_{\lambda'}^{+}$ in~$L^{2}(\Sigma)$. Since $aw_{\lambda'}=0$ and $\lambda' \ne 0$, we must have $av_{\lambda'}^{+}=0$. Therefore $v_{\lambda'}^{+}=0$ because as a classical result, the nodal set $\{v^{+}_{\lambda'}=0\}$ is of zero measure. The same argument shows that $v_{\lambda'}^{-}=0$ for $\lambda' \ne 0$. And similarly, since $\frac{1}{T}\int_{0}^{T}\pt v(t,x)\dt=v^1_0$, we have $ T \to \infty $, $0 \equiv a v^1_0 $, whence $v^1_0=0$. Therefore $v \equiv 0$, and $E(u,t_{k})\to E(v,0)=0$. 
\end{proof}

Let~$ \L $ be defined by~\eqref{eq:def-L}. Recall the undamped wave equation~\eqref{eq:undamped-wave-eq}
\begin{equation}\tag{\ref{eq:undamped-wave-eq}}
\begin{cases}
(\partial_{t}^{2}-\Delta)u=0, & \mathrm{in\ }\mathcal{D}'(\R\times \Sigma);\\
(u,\pt u)_{t=0}=(u_{0},u_{1}), & \in H^{1}(\Sigma)\times L^{2}(\Sigma),
\end{cases}
\end{equation}
and the perturbed wave equation~\eqref{eq:equation-L-zoll},
\begin{equation}\tag{\ref{eq:equation-L-zoll}}
\begin{cases}
(\partial_{t}^{2}+\L)u=0, & \mathrm{in\ }\mathcal{D}'(\R\times\Sigma);\\
(u,\pt u)_{t=0}=\left(u_{0},u_{1}\right), & \in H^{1}(\Sigma)\times L^2(\Sigma).
\end{cases}
\end{equation}

\begin{defn}
Let $ 0 \le a \in L^\infty(\Sigma) $, we say that~$ a $ observes~\eqref{eq:undamped-wave-eq} (resp.~\eqref{eq:equation-L-zoll}), if for some constant $ C > 0 $, $ T > 0 $, and every solution~$ u $ to~\eqref{eq:undamped-wave-eq} (resp.~\eqref{eq:equation-L-zoll}), the observability~\eqref{eq:ineq-observability} holds.
We say that~$ a $ observes $ \L $-eigenfunctions, if for some constant $ C > 0 $ and every $ \L $-eigenfunction $ u \in \cup_n \tilde{E}_n $, the observability~\eqref{eq:observability-L-eigenfunction} holds.
\end{defn}

We will reduce the observability of~\eqref{eq:undamped-wave-eq} to the observability of~\eqref{eq:equation-L-zoll}, and then to the observability of $ \L $-eigenfunctions. We first state some preliminaries as those used in proving Theorem~\ref{thm:Lebeau}. For $ u \in H^s(\Sigma) $ with $ s \in \R $, there exists a unique decomposition into sums of $ \L $-eigenfunctions,
\begin{equation*}
u = \sum_{n \ge 0} u_n, \quad \mathrm{with\ } u_n \in \tilde{E}_n.
\end{equation*}
Then we specify the~$ H^s(\Sigma) $ norm of~$ u $ by setting
\begin{equation*}
\|u\|_{H^s(\Sigma)}^2 = \sum_{n \ge 0} (n + 1/2)^{2s} \|u_n\|_{L^2}^2.
\end{equation*}
If we decompose the initial data $ u_i = \sum_{n \ge 0} u^i_n $, $ (i=0,1) $, with $ u^i_n \in \tilde{E}_n $, then the solution to~\eqref{eq:equation-L-zoll} is
\begin{equation*}
u(t) = \cos(t\sqrt{\L}) u_0 + \sqrt{\L}^{-1} \sin(t\sqrt{\L}) u_1
= \sum_{n \ge 0} \big( e^{it(n+1/2)} u^+_n + e^{-it(n+1/2)} u^-_n \big),
\end{equation*}
where for $ n \ge 0 $
\begin{equation*}
u^+_n + u^-_n = u^0_n, \quad i(n+1/2) (u^+_n - u^-_n) = u^1_n,
\end{equation*}
and satisfies the a priori estimate $ \|u\|_{L^\infty H^s(\Sigma)} \lesssim \|(u_0,u_1)\|_{H^s(\Sigma) \times H^{s-1}(\Sigma)} $.

\begin{lem}
\label{LEM::Reduction-Step-2} 
Let $ 0 \le a \in L^\infty(\Sigma) $, if~$ a $ observes~\eqref{eq:equation-L-zoll}, then~$ a $ observes~\eqref{eq:undamped-wave-eq}.
\end{lem}
\begin{proof}
The proof is a mimic of that of Lemma~\ref{lem:observability-reduction-to-perturbed-wave-S^d}. Write $ K = \Delta + \L $, then by the definition of~$ \L $, $ K $ is bounded on~$ L^2(\Sigma) $, with $ \|K\|_{L^2\to L^2} \le A $. Let $u$ be the solution to~\eqref{eq:undamped-wave-eq}, with initial data $ (u_0,u_1) $. There is no harm in assuming that $\int_{\Sigma} u_{0}\,\dx=0$. Decompose $u=w+v$ with
\begin{equation*}
\begin{cases}
(\pt^{2}+\L)v=0, & (v,\pt v)_{t=0}=(u_{0},u_{1});\\
(\pt^{2}+\L)w=Ku, & (w,\pt w)_{t=0}=(0,0).
\end{cases}
\end{equation*}
Now that~$ a $ observes~\eqref{eq:equation-L-zoll},
\begin{equation*}
E(u,0) = E(v,0) 
\lesssim \int_{0}^{T}\int_{\Sigma} a|\pt v|^{2}\dx\,\dt
\lesssim \int_{0}^{T}\int_{\Sigma} a|\pt u|^{2}\dx\,\dt 
+ \int_{0}^{T}\int_{\Sigma} a|\pt w|^{2}\dx\,\dt.
\end{equation*}
By Duhamel's formula, $ \pt w(t) =\int_{0}^{t}\cos((t-s)\sqrt{\L})Ku(s)\,\ds $; and we have
\begin{equation*}
\|\pt w(t)\|_{L^{2}}^{2} 
\lesssim \int_{0}^{T}\|u(s)\|_{L^{2}}^{2}\ds 
\lesssim \|(u_{0},u_{1})\|_{L^{2}\times H^{-1}}^{2}
\lesssim \|([u_{0}],u_{1})\|_{\mathcal{H}^{0}\times H^{-1}}^{2}.,
\end{equation*}
where the last inequality is because~$ u_0 $ has no zero frequency. Combine the estimates above, we obtain a weak observability, with a compact remainder term on the right hand side
\begin{equation*}
\frac{1}{2} \|([u_{0}],u_{1})\|_{\mathcal{H}^{1}\times L^{2}}^{2}
=E(u,0)
\lesssim \int_{0}^{T}\int_{M}a|\pt u|^{2}\dx\,\dt 
+ \|([u_{0}],u_{1})\|_{\mathcal{H}^{0}\times H^{-1}}^{2}.
\end{equation*}
To remove the remainder term and prove the (strong) observability, we appeal to the uniqueness-compactness argument originally due to Bardos-Lebeau-Rauch~\cite{B-L-R}. It is an argument by contradiction that carries out as follows. Suppose that the observability of~\eqref{eq:undamped-wave-eq} does not hold, then there exists a sequence of initial data $(u_{0}^{n},u_{1}^{n})\in H^{1}(\Sigma)\times L^{2}(\Sigma)$ such that, $ \int_{\Sigma} u^n_0\,\dx = 0 $, and as $ n\to\infty $
\begin{equation*}
E(u^{n},0)=\frac{1}{2}\|([u_{0}^{n}],u_{1}^{n})\|_{\mathcal{H}^{1}\times L^{2}}^{2}=1,
\quad  \int_{0}^{n}\int_{\Sigma} a|\pt u^{n}|^{2}\dx\,\dt = o(1).
\end{equation*}
where~$u^{n}$ are the corresponding solutions~\eqref{eq:undamped-wave-eq}. By Rellich's compact injection theorem, up to a subsequence, we assume that, for some $([u_{0}],u_{1})\in\mathcal{H}^{1}\times L^{2}$,
\begin{enumerate}[nosep]
\item $ ([u_{0}^{n}],u_{1}^{n}) \rightharpoonup ([u_{0}],u_{1}) $ weakly in $ \mathcal{H}^{1}(\Sigma) \times L^{2}(\Sigma) $;
\item $ ([u_{0}^{n}],u_{1}^{n}) \to ([u_{0}],u_{1}) $ strongly in $ \mathcal{H}^{0}(\Sigma) \times H^{-1}(\Sigma) $.
\end{enumerate}
Passing $n\to\infty$ in the weak observability,
\begin{equation*}
1=E(u^{n},0)
\le o(1)+\|([u_{0}^{n}],u_{1}^{n})\|_{\mathcal{H}^{0}\times H^{-1}}^{2}
\to \|([u_{0}],u_{1})\|_{\mathcal{H}^{0}\times H^{-1}}^{2}.
\end{equation*}
Therefore, we will get a contradiction by showing that the right hand side vanishes. 

To show this, we observe that $E(u^{n},t) = E(u^{n},0) = \frac{1}{2} \|([u_{0}^{n}],u_{1}^{n})\|_{\mathcal{H}^{1}\times L^{2}}^{2}$ is uniformly bounded in~$t$ and~$n$. Therefore~$[u^{n}]$ is bounded in~$L^{\infty}(\R,\mathcal{H}^{1}(\Sigma))$ and~$\pt u^{n}$ is bounded in~$L^{\infty}(\R,L^{2}(\Sigma))$. Moreover $\int_{\Sigma}u^{n}(t,x)\,\dx$ is bounded in $L_{\loc}^{\infty}(\R_{t})$ (and is of order $O(t)$). Consequently~$u^{n}$ is bounded in $L_{\loc}^{\infty}(\R,H^{1}(\Sigma))$. The theorem of Ascoli and the compact injection theorem of Rellich show that, up to a subsequence, there exists a $(u,v)\in C(\R,L^{2}(\Sigma))\times L^{\infty}(\R,L^{2}(\Sigma))$, such that 
\begin{enumerate}
\item $u^{n}\to u$ strongly in $L_{\loc}^{\infty}(\R,L^{2}(\Sigma))$ ;
\item $u^{n} \rightharpoonup u$ respect to the weak-$\ast$ topology of $L_{\loc}^{\infty}(\R,H^{1}(\Sigma))$;
\item $\pt u^{n}\rightharpoonup v$ with respect to the weak-$\ast$ topology
of $L^{\infty}(\R,L^{2}(\Sigma))$. 
\end{enumerate}
Pass to the limit in the sense of distribution, we see that~$u$ satisfies~\eqref{eq:undamped-wave-eq}, with in particular $v=\pt u$. Therefore $\pt(\pt u^{n})=\Delta u^{n}$ is bounded in $L^{\infty}(\R,H^{-1}(\Sigma))$, so that $\pt u\in C(\R,H^{-1}(\Sigma))$, and $u\in C(\R,L^{2}(\Sigma))\cap C^{1}(\R,H^{-1}(\Sigma))$. However, since $(u_{0},u_{1})\in H^{1}(\Sigma)\times L^{2}(\Sigma)$, there exists a solution in the $C(\R,H^{1}(\Sigma))\cap C^{1}(\R,L^{2}(\Sigma))$. By the uniqueness of the solution in $C(\R,L^{2}(M))\cap C^{1}(\R,H^{-1}(M))$, these two solutions must coincide. Therefore, it is legitimate to talk about the energy of~$u$, which is conserved $E(u,t)\equiv E(u,0)$.
On the other hand, since $a\pt u=0$ in~$\mathcal{D}'(\Sigma)$, $u$ should also satisfy the damped wave equation~\eqref{eq:Damped-Wave-Equation}. Then Proposition~\ref{prop:Weak-Stability-Property-Criterion} shows that the energy $E(u,t)$ must decay to zero as $t\to+\infty$. Hence $E(u,0)=0$, i.e.~$([u_{0}],u_{1})=(0,0)$. 
\end{proof}

\begin{lem}
Let $ 0 \le a \in L^\infty(\Sigma) $, if~$ a $ observes $ \L $-eigenfunctions, then~$ a $ observes~\eqref{eq:equation-L-zoll}.
\end{lem}
\begin{proof}
Recall that a solution to~\eqref{eq:equation-L-zoll} is of the form
\begin{equation*}
u(t) = \sum_{n \ge 0} \big( e^{it(n+1/2)} u^+_n + e^{-it(n+1/2)} u^-_n \big),
\end{equation*}
where $ u^\pm_n \in \tilde{E}_n $. Now that~$ a $ observes $ \L $-eigenfunctions, which implies
\begin{equation*}
\|a^{1/2} u^\pm_n\|_{L^2(\Sigma)} \gtrsim \|u^\pm_n\|_{L^2(\Sigma)},
\end{equation*}
we have, by the orthogonality of $ \{e^{\pm i(n+1/2)t}\}_{n \in \N} $ in $ L^2([0,2\pi]) $, and a similar argument to that of Lemma~\ref{lem:Observability-Spherical-Harmonics},
\begin{align*}
\int_{0}^{2\pi}\int_{\Sigma} a|\pt u|^{2}\dx\,\dt & =\int_{\Sigma} a(x) \int_{0}^{2\pi} \Big|\sum_{n \ge 0} (n+1/2) (e^{it(n+ 1/2)}u_{n}^{+} - e^{-it(n+1/2)} u_{n}^{-})\Big|^{2}\dt\,\dx\\
& = 2\pi \int_{\Sigma} a(x) \sum_{n \ge 0} \big|(n+1/2)u_{n}^{+}|^{2}+|(n+1/2)u_{n}^{-}\big|^{2}\dx\\
& = 2\pi \sum_{n \ge 0} (n+1/2)^2 \int_{\Sigma} a(x) \big(|u_{n}^{+}|^{2} + |u_{n}^{-}|^{2}\big) \dx \\
& \gtrsim \sum_{n \ge 0} (n+1/2)^2 \int_{\Sigma} \big(|u_{n}^{+}|^{2} + |u_{n}^{-}|^{2}\big) \dx \\
& \gtrsim E(u,0).
\end{align*}
\end{proof}

\subsection{Observability of $ \L $-Eigenfunctions}

This sections aims to prove the observability of $ \L $-eigenfunctions, which concludes Theorem~\ref{thm:Stability-Zoll-non-GCC}.
\begin{prop}
\label{prop:observability-L-eigenfunction}
Let~$ \Sigma $ be a Zoll surface of revolution, then $ a(x) = 1_{\Sigma^+}(x) $ observes $ \L $-eigenfunctions.
\end{prop}

We prove this proposition with an argument by contradiction. If the observability of $ \L $-eigenfunctions does not hold, then there exists a sequence of $ \L $-eigenfunctions $ u_{n_m} \in \tilde{E}_{n_m} $ such that, as $ m \to \infty $,
\begin{equation*}
\|u_{n_m}\|_{L^2(\Sigma)} = 1, \quad \|1_{\Sigma^+} u_{n_m}\|_{L^2(\Sigma)} = o(1).
\end{equation*}
If $ \{n_m\}_{m \ge 0} $ is bounded, then $ \tilde{E} := \oplus_{m \ge 0} \tilde{E}_{n_m} $ is a finite dimensional vector subspace of~$ L^2(\Sigma) $, consisting only of low frequencies, on which the estimate holds, for any $ N > 0 $,
\begin{equation*}
\|u\|_{L^2(\Sigma)} \lesssim \|u\|_{H^{-N}(\Sigma)}.
\end{equation*}
Therefore, $ (\tilde{E},\|\cdot\|_{L^2(\Sigma)}) $ is relatively compact, and the bounded sequence~$ \{u_{n_m}\}_{m \ge 0} $ admits a limit point $ u \in \tilde{E} $, that is, $ u_{n_m} \to u $ in~$ L^2(\Sigma) $, and hence $ 1_{\Sigma^+} u_{n_m} \to 1_{\Sigma^+} u $ in~$ L^2(\Sigma) $. Consequently,
\begin{equation*}
\|u\|_{L^2(\Sigma)} = 1, \quad \|1_{\Sigma^+} u\|_{L^2(\Sigma)} = 0.
\end{equation*}
However this is impossible, for~$ u $ is a finite sum of Laplacian eigenfunctions, which does not vanish only on a set of zero measure.

We are left to consider the case where~$ \{n_m\}_{m \ge 0} $ is unbounded. Up to a subsequence, we may assume that~$ n_m $ increases to~$ \infty $. For simplicity of notation, we drop the~$ m $ subindex, and write $ n = n_m $, and introduce the semiclassical parameter
\begin{equation*}
h = (n + 1/2)^{-1}.
\end{equation*}
We then write $ u(h) = u_{n} $, which satisfies $ \L u(h) = h^{-2} u(h) $, and consequently
\begin{equation}
\label{eq:equation-Helmholtz-u(h)}
(- h^2 \Delta -1) u(h) = - h^2 K u(h) = O(h^2)_{L^2(\Sigma)}.
\end{equation}

\subsubsection{Concentration of $\L$-Eigenfunctions}
\label{sub:Concentration-Eigenfunction}

We study the semiclassical measures of the sequence~$ u(h) $ and show that it concentrates on the equator. This argument is rather standard, we refer to, for example~\cite{Burq-1}, see also~\cite{Zworski}. We recall the definition of the semiclassical measure and some of its basic properties in Appendix~\ref{sec:SDM}. 

We extract a subsequence if necessary, and assume in addition that~$ u(h) $ is pure (see Remark~\ref{remark:pure} for the definition).

\begin{prop}
Let~$ \mu $ be the $ h $-semiclassical measure of~$ u(h) $, then
\begin{equation}
\label{eq:u(h)-semiclassical-measure-concentration}
\supp\mu
\subset S^{*}\Sigma\cap\{x = 0, \xi = 0\}
=\big\{(0,\varphi,0,\pm1):\varphi\in \S^{1}\big\}.
\end{equation}
\end{prop}
\begin{proof}
Recall that $ u(h) $ satisfies the equation $ (- h^2 \Delta -1) u(h) = O(h^2)_{L^2(\Sigma)} $. The principal symbol of $ -h^2 \Delta - 1 $ (in the semiclassical sense) is $ p(x,\xi) = g^{-1}_x(\xi,\xi) - 1$, where~$ g^{-1} $ is the inverse matrix of~$ g $. By Theorem~\ref{thm:prop-sm}, 
\begin{equation*}
\supp \mu \subset T^*M \cap \{p(x,\xi) = 0\} = S^*M,
\quad H_p \mu = 0.
\end{equation*} 
Now that~$ H_p $ generates the (co)-geodesic flow on~$ S^*M $, we see that~$ \mu $ is invariant by the geodesic flow. Moreover, our hypothesis $ \|1_{\Sigma^+} u(h)\| = o(1)_{L^2(\Sigma)} $ implies that 
\begin{equation*}
\supp \mu \cap T^*\Sigma^+ = \emptyset.
\end{equation*}
Recall that all geodesics enter~$ \Sigma^+ $, except for the equator,
\begin{equation*}
\supp \mu \subset S^*M \backslash \bigcup_{t \in \R} e^{tH_p} S^*\Sigma^+ = S^{*}\Sigma\cap\{x = 0, \xi = 0\}.
\end{equation*}
We conclude by a direct calculation, using $ g|_{\Gamma} = \dx^2 + \d\varphi^2 $.
\end{proof}

\begin{cor}
Let $\epsilon>0$, and $\chi_{\epsilon}\in C_{c}^{\infty}(\R)$ be such that $ 1_{[-\epsilon,\epsilon]} \le \chi_\epsilon \le 1_{[-2\epsilon,2\epsilon]} $. Then
\begin{equation*}
u(h) = \chi_{\epsilon}(1-h^{2}D_{\varphi}^{2}) u(h) + o(1)_{L^2(\Sigma)},
\end{equation*}
where $ \chi_\epsilon(1-h^2D_\varphi^2) $ is defined by functional calculus, and is thus of (semiclassical) principal symbol $ \chi_\epsilon(1-\theta^2) $ (see for example~\cite{D-S}).
\end{cor}
\begin{proof}
Let $ v(h) = u(h) - \chi_{\epsilon}(1-h^{2}D_{\varphi}^{2}) u(h) $, which is pure. Now that~$ D_\varphi $ commute with~$ -\Delta $, by~\eqref{eq:equation-Helmholtz-u(h)} we see that~$ v(h) $ satisfies $ (-h^{2}\Delta-1)v(h) = O(h^{2})_{L^{2}} $.
Therefore~$ v(h) $ is $ h $-oscillating by Example~\ref{example:h-oscillating}. And by Proposition~\ref{prop:h-oscillating-imply-Ltwo-convergence}, to conclude, it suffices to show that the semiclassical measure~$ \nu $ of $ v(h) $ vanishes. Indeed, $ \nu = \big( 1 - \chi_{\epsilon}(1-\theta^2) \big)^2 \mu = 0 $, since~$ \mu  $ is supported in $ 1 - \theta^2 = 0 $.
\end{proof}

As a consequence, in particular, for any $ \epsilon > 0 $, when~$ h $ is sufficiently small,
\begin{equation*}
\| u(h) - \chi_{\epsilon}(1-h^{2}D_{\varphi}^{2}) u(h) \|_{L^2(\Sigma)} \le \epsilon.
\end{equation*}
Fixing a sequence of $ \epsilon \to 0 $, we can find a sequence of~$ h = h_\epsilon \to 0$, so that
\begin{equation}
\label{eq:decomposition-u(h)-L-eigenfunction}
u(h) = \sum_{k \in Z_n(\epsilon)} e^{ik\varphi} \tilde{w}_{n,k} + O(\epsilon)_{L^2(\Sigma)},
\end{equation}
where $ Z_n(\epsilon) = \{k \in \Z : |1-h^2k^2| \le \epsilon \} $, and $ \tilde{w}_{n,k} \in \tilde{A}_{n,k} $.

For later convenience, we introduce the notion of admissible sequences.
\begin{defn}
\label{def:admissible}
A 4-tuple $ (\epsilon,h,k,\tilde{w}) $ is called admissible if
\begin{enumerate}[nosep]
\item $ \epsilon > 0 $, $ h = (n+1/2)^{-1} $ for some $ n \in \N $;
\item $ k \in Z_n(\epsilon) $, $ \tilde{w} \in \tilde{A}_{n,k} $.
\end{enumerate} 
A sequence of 4-tuple $ (\epsilon,h,k,\tilde{w}) $ (where by an abuse of notation, we omit the index of the sequence for simplicity) is called admissible if
\begin{enumerate}[nosep]
\item each term of the sequence is an admissible 4-tuple;
\item $ \epsilon \to 0 $, $ h \to 0 $.
\end{enumerate}
\end{defn}

\subsubsection{Reduction to Observability of 1-D Stationary Schr\"{o}dinger Equation} \label{sub:Reduction-to-1-D-Harmonic-Oscillator}

\begin{prop}
\label{prop:observability-schrodinger}
There exists $ \epsilon_0 > 0 $, $ h_0 > 0 $ and $ C > 0 $, such that for $ 0 < \epsilon < \epsilon_0 $, $ 0 < h < h_0 $, if a 4-tuple $ (\epsilon,h,k,\tilde{w}) $ is admissible, then we have the following observability,
\begin{equation*}
\|1_{x>0} \tilde{w}\|_{L^2(\rho^2\dx)} \ge C \|\tilde{w}\|_{L^2(\rho^2\dx)}.
\end{equation*}
\end{prop}

If this proposition is proven, then we can finish the proof of Proposition~\ref{prop:observability-L-eigenfunction}, and thus prove Theorem~\ref{thm:Stability-Zoll-non-GCC}. Indeed, we use the decomposition~\eqref{eq:decomposition-u(h)-L-eigenfunction}, \eqref{eq:Ltwo-decomposition-Sigma}, and the orthogonality of $ \{e^{ik\varphi}\}_{k\in\Z} $ in~$ L^2(\d\varphi,S^1) $, when $ \epsilon $ and $ h = h_\epsilon $ are sufficiently small,
\begin{align*}
\|1_{\Sigma^+}u(h)\|_{L^2(\Sigma)}^2
& \gtrsim \big\|1_{\Sigma^+} \sum_{k \in Z_n(\epsilon)} e^{ik\varphi} \tilde{w}_{n,k} \big\|_{L^2(\Sigma)}^2 - \epsilon^2 \\
& \gtrsim \int_{\R} 1_{x>0}  \int_{\S^1} \big| \sum_{k \in Z_n(\epsilon)} e^{ik\varphi} \tilde{w}_{n,k} \big|^2 \,\d\varphi \,\rho^2\dx - \epsilon^2 \\
& \gtrsim \int_{\R} 1_{x>0} \sum_{k \in Z_n(\epsilon)} |\tilde{w}_{n,k}|^2 \rho^2\dx - \epsilon^2 \\
& \gtrsim \int_{\R} \sum_{k \in Z_n(\epsilon)} |\tilde{w}_{n,k}|^2 \rho^2\dx - \epsilon^2 \\
& \gtrsim \int_{\R} \int_{\S^1} \big| \sum_{k \in Z_n(\epsilon)} e^{ik\varphi} \tilde{w}_{n,k} \big|^2 \,\d\varphi \,\rho^2\dx - \epsilon^2 \\
& \gtrsim \big\|\sum_{k \in Z_n(\epsilon)} e^{ik\varphi} \tilde{w}_{n,k} \big\|_{L^2(\Sigma)}^2 - \epsilon^2 \\
& \gtrsim \|u(h)\|_{L^2(\Sigma)}^2 - \epsilon^2 \\
& \gtrsim 1 - \epsilon^2,
\end{align*}
which contradicts to our hypothesis that $ \|1_{\Sigma^+} u(h)\|_{L^2(\Sigma)} = o(1) $ as $ h \to 0 $.

Before proving Proposition~\ref{prop:observability-schrodinger}, we observe that if~$ \tilde{w} \in \tilde{A}_{n,k} $, then~$ \tilde{w} $ satisfies a one dimensional semiclassical stationary Schr\"{o}dinger equation,
\begin{equation}
\label{eq:equation-w-A(n,k)}
(-h^{2}\partial_{x}^{2}+V)\tilde{w} = E \tilde{w} + O(h^2)_{L^2\to L^2} \tilde{w},
\end{equation}
where the potential $ V = 1-\rho^{2} $ satisfies $0\le V<1=\lim_{|x|\to\infty}V(x)$, and
\begin{equation*}
V = cx^2 + O(x^3) \quad \mathrm{near\ } x=0,
\end{equation*}
recalling that $ c = -r''(\ell_0)/2 > 0 $; while the energy
\begin{equation*}
E = 1 - h^2 k^2,
\end{equation*}
satisfies by Corollary~\ref{cor:rotation-number-less-than-energy} and Proposition~\ref{prop:spectral-zoll} the estimate
\begin{equation*}
E = 1 - \lambda^{-2} k^2 + h^2 (\lambda^2 - h^{-2}) \lambda^{-2}k^2 \gtrsim -h^2.
\end{equation*}
To obtain~\eqref{eq:equation-w-A(n,k)}, we write $ \tilde{w} = \sum_{\lambda^2 \in I_n} w_{\lambda,k} $ with $ w_{\lambda,k} \in A_{\lambda,k} $, then by Proposition~\ref{prop:Geometry-Zoll-Surface-Revolution}, $ w_{\lambda,k} $ satisfies
\begin{equation}
\label{eq:equation-w(lambda,k)}
(-h^2 \partial_x^2 + V) w_{\lambda,k} 
= E w_{\lambda,k} + h^2 (\lambda^2 - h^{-2})\rho^2 w_{\lambda,k}
= E w_{\lambda,k} + O(h^2) \rho^2 w_{\lambda,k}.
\end{equation}
It remains to sum up $ w_{\lambda,k} $, and use the orthogonality by Corollary~\ref{cor:orthogonality-A(lamnbda,k)} to obtain the estimate for the remainder term (be careful that the constant~$ O(h^2) $ varies for different $ w_{\lambda,k} $, and cannot be moved to the front of the summation)
\begin{align*}
\big\| \sum_{\lambda^2 \in I_n} O(h^2) \rho^2 w_{\lambda,k} \big\|_{L^2}^2
& \le \|\rho\|_{L^\infty}^2 \big\| \sum_{\lambda^2 \in I_n} O(h^2) w_{\lambda,k} \big\|_{L^2(\rho^2\dx)}^2 
\lesssim \sum_{\lambda^2 \in I_n} \|O(h^2) w_{\lambda,k} \|^2_{L^2(\rho^2\dx)} \\
& \lesssim h^4 \sum_{\lambda^2 \in I_n} \|w_{\lambda,k} \|^2_{L^2(\rho^2\dx)}
\lesssim h^4 \big\| \sum_{\lambda^2 \in I_n} w_{\lambda,k} \big\|_{L^2(\rho^2\dx)}^2
\lesssim h^4 \|\tilde{w}\|_{L^2}^2
\end{align*}

\begin{proof}[Proof of Proposition~\ref{prop:observability-schrodinger}]
A first consequence of~\eqref{eq:equation-w-A(n,k)} is that, by an Lithner-Agmon type estimate, $ \tilde{w} $ decays exponentially at infinity, so that the weight~$ \rho^2 $ can be dropped (which will be done by Corollary~\ref{cor:Use-Lithner-Agmon-Lose-Weight}), and we are left to prove the observability,
\begin{equation*}
\|1_{x>0} \tilde{w}\|_{L^2} \gtrsim \|\tilde{w}\|_{L^2}.
\end{equation*}
Then we proceed with an argument by contradiction. Suppose that this observability is not true, then we can find an admissible sequence of $ (\epsilon,h,k,\tilde{w}) $ which violates the observability in the sense that
\begin{equation*}
\|1_{x>0} \tilde{w}\|_{L^2} / \|\tilde{w}\|_{L^2} = o(1).
\end{equation*}
Now that $ \tilde{w} $ satisfies~\eqref{eq:equation-w-A(n,k)}, and as we have seen, since $ k \in Z_n(\epsilon) $, the energy~$ E $ satisfies
\begin{equation*}
-h^2 \lesssim E \le \epsilon = o(1),
\end{equation*}
we may assume that, up to a subsequence, either $ E = O(h) $, or $ E \gg h $. We will show that, by Proposition~\ref{prop:E=O(h)} and Proposition~\ref{prop:E-gg-h}, neither of these two cases is possible. This contradiction then finishes the proof.
\end{proof}

\subsubsection{Some Lithner-Agmon Type Estimates}
\label{sec:Lithner-Agmon-Estimates}

In this section we prove some estimates of Lither-Agmon Type, originally due to Lither~\cite{Lithner} and Agmon~\cite{Agmon}. The argument we used here comes from~\cite{D-S,Helffer}.
Let 
\begin{equation*}
P(\tau)=-h(\tau)^{2}\partial_x^2 + V(x;\tau)
\end{equation*}
be a Schr\"{o}dinger operator on~$\R$, where the parameter~$h(\tau)$ and the potential~$V(\cdot;\tau) \in C(\R) \cap L^\infty(\R) $ both depend on $\tau\to 0$. We will consider the following two cases:
\begin{enumerate}[nosep]
\item $h(\tau) \equiv 1$ does not depend on $\tau$, then we get a classical Schr\"{o}dinger operator;
\item $h(\tau) \equiv \tau \to 0$, and we get a semiclassical Schr\"{o}dinger operator.
\end{enumerate}  
We will estimate the solution $ u $ to the equation
\begin{equation}
\label{eq:equation-Schrodinger-Lithner-Agmon}
P(\tau)u=E(\tau)u+f(\tau),
\end{equation}
where $ E(\tau) \in \R $, $ f \in C(\R) \cap L^2(\R) $. To do this, we define the Lithner-Agmon distance, for $ x_1, x_2 \in \R $,
\begin{equation*}
d(x_{1},x_{2};\tau)=\Big|\int_{x_{1}}^{x_{2}}\left(V(x;\tau)-E(\tau)\right)_{+}^{1/2}\dx\Big|.
\end{equation*}
For $\varepsilon > 0$, $ R > 0 $, let 
\begin{equation*}
\Phi^{\varepsilon}(x;\tau)=(1-\varepsilon)d(x,0;\tau),
\quad
\Phi_{R}^{\varepsilon}(x;\tau)=\chi_{R}(\Phi^{\varepsilon}(x;\tau)),
\end{equation*}
where $\chi_{R}(t)=1_{t\le R}(t)t+1_{t>R}(t)R$. 

We make the following assumption.

\paragraph{\textbf{Assumption}}
For all $\varepsilon>0$, there exist $\tau_{\varepsilon}>0$, $\delta_{\varepsilon}>0$, $R_{\varepsilon}>0$, $C_{\varepsilon}>0$, such that for $0<\tau<\tau_{\varepsilon}$, if $|x|\ge R_{\varepsilon}$, then $V(x;\tau)\ge E(\tau)+\delta_{\varepsilon}$; if $|x|\le R_{\varepsilon}$, then $|V(x;\tau)-E(\tau)|<C_{\varepsilon}$, and $\Phi^{\varepsilon}(x;\tau)\le\varepsilon$.

This assumption implies that $\Phi^{\varepsilon}(x;\tau)\to\infty$ as $|x|\to\infty$, uniformly for~$\tau$ sufficiently small. Therefore $\Phi_{R}^{\varepsilon}$ is constant, equaling to~$R$, for~$|x|$ sufficiently large.

We will drop the parameter~$ \tau $ for simplicity. The following proposition comes from~\cite{D-S}.

\begin{prop}
Let $u\in C_{c}^{2}(\R)$ and let $\Phi\in\Lip_{\loc}(\R)$ be real valued, then the following identity holds.
\begin{equation}
\label{eq:Lithner-Agmon-Identity}
h^{2}\int_{\R}|(e^{\Phi/h}u)'|^{2}\dx  
+\int_{\R}(V-|\Phi'|^{2})e^{2\Phi/h}|u|^{2}\dx
=\Re\int_{\R}e^{2\Phi/h}Pu\bar{u}\,\dx.
\end{equation}
\end{prop}

Suppose now that the phase~$\Phi$ is constant for~$|x|$ large, and suppose $ u \in C^2(\R) \cap D(P) $ with $ D(P)=\big\{ w\in L^{2}(\R):Vw\in L^{2}(\R), w''\in L^{2}(\R)\big\} $. Set $u_{R}(x)=\chi(x/R)u(x)$, with $\chi\in C_{c}^{\infty}(\R)$. Therefore $u_{R}\in C_{c}^{2}(\R)$, and the previous proposition applies.
\begin{equation*}
h^{2}\int_{\R}|(e^{\Phi/h}u_{R})'|^{2}\dx  +\int_{\R}(V-|\Phi'|^{2})e^{2\Phi/h}|u_{R}|^{2}\dx
=\Re\int_{\R}e^{2\Phi/h}Pu_{R}\overline{u_{R}}\dx.
\end{equation*}
By the Lebesgue dominated convergence theorem, $u_{R}\to u$ and $Vu_{R}\to Vu$ both strongly in~$L^{2}(\R)$ as $ R\to \infty $. Now that $ u \in D(P) $, $Pu_{R}\to Pu$ in~$L^{2}(\R)$. Now that~$ \Phi $ being constant for large~$ |x| $, we can pass to the limit on each side of the identity above, and prove the following corollary.
\begin{cor}
If $u\in C^2(\R) \cap D(P)$, and~$\Phi$ is constant for large~$|x|$, then the identity~\eqref{eq:Lithner-Agmon-Identity} holds.
\end{cor}

Now let~$ u \in C(\R) \cap L^2(\R) $ be a solution to~\eqref{eq:equation-Schrodinger-Lithner-Agmon}, then $ u \in C^2(\R) \cap D(P) $, and the corollary applies,
\begin{equation}
\label{eq:Lithner-Agmon-Identity-epsilon-R-for-Eigenfunction}
\begin{split}
h^{2}\int_{\R}|(e^{\Phi_{R}^{\varepsilon}/h}u)'|^{2}\dx & +\int_{\R}(V-E-|(\Phi_{R}^{\varepsilon})'|^{2}) e^{2\Phi_{R}^{\varepsilon}/h}|u|^{2}\dx  \\
& = \Re\int_{\R}e^{2\Phi_{R}^{\varepsilon}/h}f\bar{u}\,\dx 
\le A_\varepsilon \|e^{\Phi_{R}^{\varepsilon}/h}u\|_{L^2}^{2}+C_{\varepsilon}\|e^{\Phi_{R}^{\varepsilon}/h}f\|_{L^2}^{2}.
\end{split}
\end{equation}
where $ A_\varepsilon = (1-(1-\varepsilon)^{2})\delta_{\varepsilon} $.
For $0<\tau<\tau_{\varepsilon}$ and $|x|\ge R_{\varepsilon}$, by the definition of~$ \Phi_{R}^{\varepsilon} $,
\begin{equation*}
V(x)-E-|(\Phi_{R}^{\varepsilon})'|^{2} \ge(1-(1-\varepsilon)^{2})(V(x)-E)
\ge(1-(1-\varepsilon)^{2})\delta_{\varepsilon} = A_\varepsilon.
\end{equation*}
Separating domain the integrals in~\eqref{eq:Lithner-Agmon-Identity-epsilon-R-for-Eigenfunction} into two parts, $|x|\ge R_{\varepsilon}$ and $|x|<R_{\varepsilon}$, we get
\begin{align*}
h^{2}\int_{\R}|(e^{\Phi_{R}^{\varepsilon}/h}u)'|^{2}\dx 
& +A_\varepsilon \int_{|x|\ge R_{\varepsilon}}e^{2\Phi_{R}^{\varepsilon}/h}|u|^{2}\dx
-C_{\varepsilon}\|e^{\Phi_{R}^{\varepsilon}/h}f\|_{L^2}^{2}\\
& \le \big(\|V(x)-E+(\Phi_{R}^{\varepsilon})'\|_{L^{\infty}(|x|\le R_{\varepsilon})}+A_\varepsilon\big)
 \int_{|x|\le R_{\varepsilon}}e^{2\Phi_{R}^{\varepsilon}/h}\left|u\right|^{2}\dx\\
& \le C_{\varepsilon}\int_{|x|\le R_{\varepsilon}}e^{2\Phi_{R}^{\varepsilon}/h}\left|u\right|^{2}\dx.
\end{align*}
Adding $A_\varepsilon \int_{|x|\le R_{\varepsilon}} e^{2\Phi_{R}^{\varepsilon}/h}|u|^{2}\dx$ to each side of the inequality, we get
\begin{align*}
h^{2}\int_{\R}|(e^{\Phi_{R}^{\varepsilon}/h}u)'|^{2}\dx&+ A_\varepsilon \int_{\R}e^{2\Phi_{R}^{\varepsilon}/h}|u|^{2}\dx  -C_{\varepsilon}\|e^{\Phi_{R}^{\varepsilon}/h}f\|_{L^2}^{2}\\
& \le C_\varepsilon\int_{|x|\le R_{\varepsilon}} e^{2\Phi_{R}^{\varepsilon}/h}|u|^{2}\dx
\le C_\varepsilon\sup_{|x|\le R_{\varepsilon}} (e^{2\Phi_{R}^{\varepsilon}/h})\|u\|_{L^{2}}^{2}
\le C_\varepsilon e^{2\varepsilon/h}\|u\|_{L^{2}}^{2}.
\end{align*}
This proves the following proposition.

\begin{prop}[Inhomogeneous Lithner-Agmon Estimate]
\label{prop:Lithner-Agmon-Estimate-Inhomogeneous}
Under the assumptions above, for each $\varepsilon>0$, there exists $\tau_{\varepsilon}>0$ and $C_\varepsilon>0$, such that for $0<\tau<\tau_{\varepsilon}$, and $R>0$, the following estimate holds
\begin{equation*}
\| h(e^{\Phi_{R}^{\varepsilon}/h}u)'\|_{L^{2}}^{2}
+\| e^{\Phi_{R}^{\varepsilon}/h}u\|_{L^{2}}^{2}
\le C_\varepsilon\big(e^{2\varepsilon/h} \|u\|_{L^2}^2 + \|e^{\Phi_{R}^{\varepsilon}/h}f\|_{L^2}^{2}\big).
\end{equation*}
\end{prop}

The following two corollaries are important.
\begin{cor}[Homogeneous Lithner-Agmon Estimate]
\label{prop:Lithner-Agmon-Estimate-Homogeneous}
If $f=0$, then we obtain the usual (homogeneous) Lithner-Agmon estimate,
\begin{equation*}
\| h(e^{\Phi_{R}^{\varepsilon}/h}u)'\|_{L^{2}}^{2}
+\| e^{\Phi_{R}^{\varepsilon}/h}u\|_{L^{2}}^{2}
\le C_\varepsilon e^{2\varepsilon/h}\|u\|_{L^2}^2.
\end{equation*}
Observe that, the right hand side of this estimate does not depend on~$R$, we are thus allowed to let $R\to\infty$, and get a finer estimate,
\begin{equation*}
\| h(e^{\Phi^{\varepsilon}/h}u)'\|_{L^{2}}^{2}
+\| e^{\Phi^{\varepsilon}/h}u\|_{L^{2}}^{2}
\le C_\varepsilon e^{2\varepsilon/h}\|u\|_{L^2}^2.
\end{equation*}
\end{cor}

\begin{cor}
\label{Lithner-Agmon-Estimate-Inhomogeneous-Decay-in-Classical-Forbidden-Region}
Let $\chi\in L^\infty(\R)$ be supported in the interior of $\{x\in\R:\Phi^{\varepsilon}_{R}(x)=R\}$, such that $0\le \chi\le 1$, then
\begin{equation}
\|\chi h u'\|_{L^2}^{2}+\|\chi u\|_{L^2}^{2}\le C_\varepsilon\big(e^{-2(R-\varepsilon)/h}\|u\|_{L^2}^{2}+\|f\|_{L^2}^{2}\big).
\end{equation}
\end{cor}
\begin{rem}
For any $\delta>0$, we could modify the phase function $\Phi^{\varepsilon}_{R}$ to some $\tilde{\Phi}_{R}$, so that $\tilde{\Phi}_{R}\equiv R$ for $|x|\ge R_{\varepsilon}+\delta$, while $\tilde{\Phi}_{R}=\Phi_{R}^{\varepsilon}$ for $|x|\le R_{\varepsilon}$.
\end{rem}
\begin{rem}
This is a classical estimate by reversing the operator $-h^{2}\partial_{x}^{2}+V(x)$ in the classical forbidden region, when $V$ is independent of $\tau\equiv h$. For our application in Section~\ref{sub:E-gg-h}, where the potential is~$V_{E}$, it is believed that such a semiclassical analysis suffices. However, we decide to use the approach above for simplicity to avoid technique problems caused by the behavior of $V_{E}$ at faraway from the origin.
\end{rem}
\begin{proof}
Simply notice that
\begin{equation*}
\chi(e^{\Phi^{\varepsilon}_{R}/h}u)'=e^{R/h}\chi u',
\quad \chi e^{\Phi^{\varepsilon}_{R}/h}u=e^{R/h}\chi u,
\quad \|e^{\Phi^{\varepsilon}_{R}/h}f\|_{L^2}\le e^{R/h}\|f\|_{L^2}.
\end{equation*}
The rest of the proof is a straightforward application of the previous proposition.
\end{proof}

We want to apply the discussion above to an admissible 4-tuple $ (\epsilon,h,k,\tilde{w}) $ for sufficiently small~$ \epsilon $ and~$ h $. So that $\tau = h$, $P=-h^{2}\partial_{x}^{2}+V$, and $E=o(1)$, and $f = O(h^2)_{L^2\to L^2} \tilde{w}$. We are left to verify that $\tilde{w}\in D(P)$. This requires the following proposition from~\cite{Olver}.
\begin{prop}
Let $ I=(a_{-},a_{+}) \subset \R $ be a finite or infinite interval, let $f\in C^{2}(\bar{I})$ be real valued and positive, and let $g \in C(\bar{I})$ be a continuous and complex valued. Let
\begin{equation*}
F(x)=\int\big\{ f^{-1/4}(f^{-1/4})''-gf^{-1/2}\big\}\,\dx
\end{equation*}
be a primitive function of the integrand. Then in $I$ the differential
equation
\begin{equation*}
u''=(f+g)u
\end{equation*}
has twice continuously differentiable solutions of the form
\begin{equation*}
u_{\pm}(x)=f^{-1/4}(x)\exp\Big\{\pm\int f^{1/2}(x)\dx\Big\}\left(1+\varepsilon_{\pm}(x)\right)
\end{equation*}
with estimates
\begin{equation*}
\max\big\{ \left|\varepsilon_{\pm}(x)\right|,
\frac{1}{2}f^{-1/2}(x)\left|\varepsilon'_{\pm}(x)\right|\big\}
\le\exp\big\{\frac{1}{2}\mathcal{V}_{a_{\pm},x}(F)\big\}-1
\end{equation*}
provided the total variation $\mathcal{V}_{a_{\pm},x}(F)$ of~$F$ on the interval $(a_{\pm},x)$ being finite. If $g$ is real, then the solutions are real.
\end{prop}
\begin{cor}
\label{cor:Solution-Approximate}
Let $ w \in A_{\lambda,k} $, with $ \lambda > 0 $, $ k \ne 0 $, then on the interval $(R_{0},\infty)$, $w$ is, up to a multiplicative constant, of the form
\begin{equation*}
w(x)
=[V(x)-E]^{-1/4}\exp\Big\{-h\int_{0}^{x}[V(t)-E]^{1/2}\dt\Big\}(1+\varepsilon(x))
\end{equation*}
with estimates $|\varepsilon(y)|+|\varepsilon'(y)|=O(h)$.
We can do the same on $(-\infty,-R_{0})$, and consequently~$w\in H^{1}(\R)$. Since $V\in L^{\infty}(\R)$, we deduce that $w\in D(P)$. Now that~$ \tilde{w} $ is a finite sum of such~$ w_{\lambda,k} $, we deduce that $ \tilde{w} \in H^1(\R) $.
\end{cor}
\begin{proof}
Apply the previous proposition with $f=h^{-2}\left(V-E\right)$ and
$g=0.$ Then 
\[
F(x)=h\int_{c}^{x}\left[V(t)-E\right]^{-1/4}\partial_{t}^{2}\left[V(t)-E\right]^{-1/4}\dt,
\]
from which, for $x>R_{0}$,
\begin{align*}
\mathcal{V}_{x,\infty}(F)
 & \le Ch\delta^{-5/2}\big(\|{r'}\|_{L^{\infty}}^{2}+\|{r''}\|_{L^{\infty}}\big)\int_{\R}\rho^{2}(t)\,\dt =O(h)
\end{align*}
since $V(x)=1-\rho^{2}(x)=1-r^{2}\left(f(x)\right)$, $f'(x)=r\left(f(x)\right)$, and that
\begin{equation*}
\int\rho^{2}(t)\,\dt
=\frac{1}{2\pi}\int_{0}^{2\pi}\d\varphi\int_\R\rho^{2}(x)\,\dx
=\frac{1}{2\pi}\mathrm{vol}(M)<\infty.
\end{equation*} 
\end{proof}

\begin{cor}
\label{cor:Use-Lithner-Agmon-Lose-Weight}
There exists $ \epsilon_0 > 0 $, $ h_0 > 0 $, $ C > 0 $, such that for $ 0 < \epsilon < \epsilon_0 $, $ 0 < h < h_0 $, if $ (\epsilon,h,k,\tilde{w}) $ is an admissible 4-tuple, then
\begin{equation*}
C^{-1}\|\tilde{w}\|_{L^2} \le \|\tilde{w}\|_{L^2(\rho^{2}\dx)} \le C \|\tilde{w}\|_{L^2};
\end{equation*}
Suppose there exists $ \epsilon_0 > 0 $, $ h_0 > 0 $, $ C > 0 $, such that for $ 0 < \epsilon < \epsilon_0 $, $ 0 < h < h_0 $, if $ (\epsilon,h,k,\tilde{w}) $ is an admissible 4-tuple, then
\begin{equation*}
\|1_{x>0}\tilde{w}\|_{L^{2}} \ge C \|\tilde{w}\|_{L^{2}},
\end{equation*}
then there exists $ \epsilon'_0 > 0 $, $ h'_0 > 0 $, $ C' > 0 $, such that for $ 0 < \epsilon < \epsilon'_0 $, $ 0 < h < h'_0 $, if $ (\epsilon,h,k,\tilde{w}) $ is an admissible 4-tuple, then
\begin{equation*}
\|1_{x>0}\tilde{w}\|_{L^{2}(\rho^2\dx)} \ge C' \|\tilde{w}\|_{L^{2}(\rho^2\dx)}.
\end{equation*}
\end{cor}
\begin{proof}
There is no harm in assuming $ \|\tilde{w}\|_{L^2} = 1 $, and apply Corollary~\ref{Lithner-Agmon-Estimate-Inhomogeneous-Decay-in-Classical-Forbidden-Region} with 
\begin{equation*}
f = O(h^2)_{L^2\to L^2} \tilde{w} = O(h^2).
\end{equation*}
To do this, we fix $0<\varepsilon<1$ (please do not get confused with~$ \epsilon $), and fix $R>2\varepsilon$, then for some $R_{0}>0$, $\chi=1_{|x|>R_{0}}$ is supported in $\{\Phi_{R}^{\varepsilon}=R\}$. Then Corollary~\ref{Lithner-Agmon-Estimate-Inhomogeneous-Decay-in-Classical-Forbidden-Region} implies that, for some constant $ C_\varepsilon > 0 $, 
\begin{equation*}
\|1_{|x|>R_0}\tilde{w}\|_{L^2(\rho^2\dx)} \le \|1_{|x|>R_0}\tilde{w}\|_{L^2} \le C_\varepsilon h^2
\end{equation*}
Let $ \delta = \inf_{|x|<R_{0}} \rho(x) > 0 $, then 
\begin{align*}
1 = \|\tilde{w}\|_{L^2} 
& \ge \|\tilde{w}\|_{L^2(\rho^{2}\dx)}
\ge \|1_{|x|<R_{0}} \tilde{w}\|_{L^2(\rho^{2}\dx)}
\ge \delta^{-1} \|1_{|x|<R_{0}}\tilde{w}\|_{L^2} \\
& \ge \delta^{-1} (1-\|1_{|x|>R_{0}}\tilde{w}\|_{L^2}) 
 \ge \delta^{-1} (1- C_\varepsilon h^2)
 \ge \frac{1}{2} \delta^{-1},
\end{align*}
when $ h $ is sufficiently small. This proves the first statement. To prove the second statement,
\begin{align*}
\|1_{x>0}\tilde{w}\|_{L^{2}(\rho^2\dx)}/\|\tilde{w}\|_{L^{2}(\rho^2\dx)}
& \ge \|1_{R_0>x>0}\tilde{w}\|_{L^{2}(\rho^2\dx)} / \|\tilde{w}\|_{L^{2}} \\
& \ge \delta^{-1} \|1_{R_0>x>0}\tilde{w}\|_{L^{2}} / \|\tilde{w}\|_{L^{2}} \\
& \ge \delta^{-1} \big(\|1_{x>0}\tilde{w}\|_{L^{2}} - \|1_{x>R_0}\tilde{w}\|_{L^{2}} \big) / \|\tilde{w}\|_{L^{2}} \\
& \ge \delta^{-1} (C - C_\varepsilon h^2) \\
& \ge \frac{1}{2} \delta^{-1} C,
\end{align*}
when $ h $ is sufficiently small.
\end{proof}

\subsubsection{Case $E=O(h)$}
\label{sub:E=O(h)}

\begin{prop}
\label{prop:E=O(h)}
Let $ (\epsilon,h,k,\tilde{w}) $ be an admissible sequence such that $ E = O(h) $, then for some $ C > 0 $ and $ \epsilon $, $ h $ sufficiently small,
\begin{equation*}
\|1_{x>0} \tilde{w}\|_{L^2} \ge C \|\tilde{w}\|_{L^2}.
\end{equation*}
\end{prop}
\begin{proof}
We first study Laplacian eigenfunctions, rather than $ \L $-eigenfunctions for simplicity, for the latter are finite sums of the former. To do this, we let $ \lambda^2 \in I_n $, $ k \in Z_n(\epsilon) $, and $ w \in A_{\lambda,k} $. Recall that~$ w $ satisfies
\begin{equation*}
(-h^2\partial_x^2 + V) w = E w + O(h^2)_{L^\infty} w.
\end{equation*}
Up to a subsequence, we may assume that $ c^{-1/2} h^{-1} E = F + o(1) $ for some $ F \ge 0 $, and use the following scaling,
\begin{equation*}
z = c^{1/4} h^{-1/2} x, \quad V_h(z) = c^{-1/2}h^{-1} V(x),
\end{equation*}
and work under the coordinate~$ z $, and with the measure~$ \dz $. We normalize~$ w $ so that $ \|w\|_{L^2} = 1 $, and observe that it satisfies the equation
\begin{equation*}
(-\partial_z^2 + V_h) w = F w + o(1)_{L^\infty} w.
\end{equation*}
Notice that $ V_h(z) = z^2 + h^{1/2} O(z^3) $ for $ |z| \lesssim h^{-1/2} $, we apply Proposition~\ref{prop:Lithner-Agmon-Estimate-Homogeneous} with
\begin{equation*}
\tau(h) \equiv 1, \quad P(\tau) = -\partial_z^2 + V_h(z),
\quad \Phi^{\varepsilon}(z)=(1-\varepsilon)\Big|\int_{0}^{z}\big(V_{h}(t)-F-o(1)_{L^\infty}\big)_{+}^{1/2}\dt\Big|,
\end{equation*}
and get
\begin{equation*}
\|(e^{\Phi^{\varepsilon}}w)'\|_{L^{2}}^{2}
+\| e^{\Phi^{\varepsilon}}w\|_{L^{2}}^{2}
\le C_\varepsilon\|w\|_{L^{2}}^{2},
\end{equation*}
which implies 
\begin{equation*}
\| e^{\Phi^{\varepsilon}}w\|_{L^{2}}
+\|h^{1/2} e^{\Phi^{\varepsilon}} w'\|_{L^{2}}
\le C_\varepsilon\|w\|_{L^{2}}.
\end{equation*}
Indeed, 
\begin{align*}
\|h^{1/2} e^{\Phi^{\varepsilon}} w'\|_{L^{2}} 
&\le \|h^{1/2}(e^{\Phi^{\varepsilon}}w)'\|_{L^{2}}
+\|h^{1/2}(\Phi^{\varepsilon})'e^{\Phi^{\varepsilon}}w\|_{L^{2}}\\
& \le h^{1/2} C_\varepsilon \|w\|_{L^{2}}
+(1-\varepsilon)(\|V\|_{L^{\infty}}+O(h))^{1/2}\|e^{\Phi^{\varepsilon}}w\|_{L^{2}}\\
 & \le C_\varepsilon\|w\|_{L^{2}}.
\end{align*}
Since $\Phi^{\varepsilon}(z)\ge\alpha(|z|-M)$ for some $\alpha > 0$, $ M>0$ and is uniform for all small $\varepsilon$, $h$, we then have
\begin{equation*}
\|w\|_{L^{2}(|z|\ge R)}+\|h^{1/2}\partial_{z}w\|_{L^{2}(|z|\ge R)}=O(e^{-\alpha R})\|w\|_{L^{2}}.
\end{equation*}
Fix some $0<\delta<1/6$, and let $w_\chi=\chi(h^{\delta}z)w(z)$ where $\chi\in C^{\infty}_{0}$ is a cut-off function equaling to 1 near the origin. Therefore 
\begin{equation*}
w
=w_\chi + O(h^{-1/2}e^{-\alpha h^{-\delta}})_{H^{1}}
=w_\chi + O(h^{\infty})_{H^1}.
\end{equation*}
Observing that on the support of $w_\chi$, $V_{h}(z) = z^{2} + O(h^{1/2-3\delta})$, we have,
\begin{equation}
\label{eq:Equation-for-w-tilde}
\begin{split}
(-\partial_{z}^{2}+z^{2}-F) w_\chi 
& = o(1)_{L^{\infty}} w + [\partial_{z}^{2},\chi(h^{\delta}z)] w \\
& =o(1)_{L^{\infty}} w + 2h^{\delta}\chi'(h^{\delta}z) w' + h^{2\delta}\chi''(h^{\delta}z) w \\
& =o(1)_{L^{2}}.
\end{split}
\end{equation}

Let $\left\{ v_{i}\right\} _{i\in\N}$ be the complete set of normalized
eigenfunctions of the classical harmonic oscillator, $-\partial_{z}^{2}+z^{2}$, that is, $ \|v_i\|_{L^2} = 1 $, and
\begin{equation*}
(-\partial_{z}^{2}+z^{2})v_{i}=(2i+1)v_{i}.
\end{equation*}
We know that $ v_i(z) = c_i H_i(z) e^{-z^2/2} $, where~$ c_i $ is a constant of normalization, and~$ H_i $ is the $ i^{\mathrm{th}} $~Hermite polynomial. We will only use the fact that $ H_i $ is either an odd function (when $ i $ is odd), or an even function (when $ i $ is even).

We write $ w_\chi = \sum \alpha_{i}v_{i}$, and have
\begin{equation}
\label{eq:L2-norm-for-w-tilde}
\sum_{i\ge0}|\alpha_{i}|^{2}=\|w_\chi\|_{L^{2}}^{2}=\|w\|_{L^{2}}^{2}+o(1)=1+o(1).
\end{equation}
By (\ref{eq:Equation-for-w-tilde}),
\begin{equation*}
o(1)_{L^{2}}  
=(-\partial_{z}^{2}+z^{2}-F) w_\chi
=\sum_{i\ge0}(2i+1-F)\alpha_{i}v_{i},
\end{equation*}
which gives 
\begin{equation}
\sum_{i\ge0}(2i+1-F)^{2}|\alpha_{i}|^{2}=o(1).
\label{eq:Consequence-of-Eq-for-w-tilde}
\end{equation}
Let $i_{0} \in \N $ be such that $|2i_{0}+1-F|$ attains the minimum among all $|2i+1-F_{0}|$. Then for any integer $i\ne i_{0}$, $|2i+1-F_{0}|\ge|i-i_{0}|$, and hence, 
\begin{equation*}
\big\|\sum_{i\ne i_{0}}\alpha_{i}v_{i}\big\|_{L^{2}}^{2}
=\sum_{i\ne i_{0}}\left|\alpha_{i}\right|^{2}=o(1).
\end{equation*}
Combining with~\eqref{eq:L2-norm-for-w-tilde}, we have $ \alpha_{i_{0}}=1+o(1) $. And by consequence,
\begin{equation*}
w
= w_\chi + o(1)_{L^2}
=\alpha_{i_{0}}v_{i_{0}}+\sum_{i\ne i_{0}}\alpha_{i}v_{i} + o(1)_{L^2}
=\alpha_{i_{0}}v_{i_{0}}+o(1)_{L^2}
=v_{i_{0}}+o(1)_{L^2}.
\end{equation*}
Moreover, we have by~\eqref{eq:Consequence-of-Eq-for-w-tilde}, that $(2i_{0}+1-F)\left|\alpha_{i_{0}}\right|^{2}=o(1)$, which implies 
\begin{equation*}
F=2i_{0}+1.
\end{equation*}
In particular~$ i_0 $ depends only on~$ F $, not on~$ \lambda $. As a consequence, we claim that, for this admissible subsequence, which satisfies $ E = O(h) $, when $ \epsilon $ and $ h $ are sufficiently small, there can be at most one $ \lambda^2 \in I_n $, such that $ A_{\lambda,k} \ne \{0\} $. Therefore, $ \tilde{A}_{n,k} = A_{\lambda,k} $. So if $ \tilde{w} \in \tilde{A}_{n,k} $, then $ \tilde{w} = v_{i_0} + o(1)_{L^2} $. This concludes the proof, since~$ v_{i_0} $ is either an odd function, or an even function, whose $ L^2 $ norm is thus equally distributed on each side of the origin.

To prove the claim, we argue by contradiction and use the orthogonality given by Corollary~\ref{cor:orthogonality-A(lamnbda,k)}. Indeed, suppose we can find for arbitrary small~$ \epsilon $ and~$ h $ two distinct $ \lambda_1 $, $ \lambda_2 $ such that $ \lambda_i^2 \in I_n $, $ (i=1,2) $, and that $ A_{\lambda_i,k} \ne \{0\} $, then we can choose $ w_i \in A_{\lambda_i,k} $, such that $ \|w_i\|_{L^2} = 1 $. By the analysis above, we see that $ w_i = v_{i_0} + o(1)_{L^2} $. Using the orthogonality of~$ w_1 $ and~$ w_2 $ with respect to $ L^2(\rho^2\dz) $,
\begin{equation*}
0 = (w_1,w_2)_{L^2(\rho^2\dz)} = (v_{i_0},v_{i_0})_{L^2(\rho^2\dz)} + o(1) \to (v_{i_0},v_{i_0})_{L^2(\rho^2\dz)} \ne 0,
\end{equation*}
we obtain a contradiction.
\end{proof}

\subsubsection{Case $E \gg h$}
\label{sub:E-gg-h}

\begin{prop}
\label{prop:E-gg-h}
Let $ (\epsilon,h,k,\tilde{w}) $ be an admissible sequence such that $ E \gg h $, then for some $ C > 0 $, and $ \epsilon $, $ h $ sufficiently small,
\begin{equation*}
\|1_{x>0} \tilde{w}\|_{L^2} \ge C \|\tilde{w}\|_{L^2}.
\end{equation*}
\end{prop}
\begin{proof}
We use the scaling
\begin{equation*}
z=c^{1/2} E^{-1/2}x,
\quad \h=c^{1/2} E^{-1} h,
\quad V_{E}(z)=E^{-1}V(x),
\end{equation*}
and work under the~$ z $ coordinate and the measure~$ \dz $. We normalize~$ \tilde{w} $ by $ \|\tilde{w}\|_{L^2} = 1 $, and observe that it satisfies the equation
\begin{equation}
\label{eq:Harmonic-Oscillator-E-gg-h}
(-\h^{2}\partial_{z}^{2}+V_{E}) \tilde{w}
= \tilde{w} + O(h^2/E) \rho^2 \tilde{w}
= \tilde{w} + o(\h)_{L^2}.
\end{equation}
Let $\chi\in C^\infty_{c}(\R)$ be equal to~1 in a neighbourhood of $|z|\le 1$, and $0\le\chi\le 1$, then we apply Corollary~\ref{Lithner-Agmon-Estimate-Inhomogeneous-Decay-in-Classical-Forbidden-Region} and the remark after it,
\begin{equation*}
\|(1-\chi) \h \tilde{w}'\|_{L^2}^{2} + \|(1-\chi) \tilde{w}\|_{L^2}^{2} 
= O(\h^{\infty}) \|\tilde{w}\|_{L^2}^{2} + o(h^{2}) = o(\h^{2}),
\end{equation*}
which implies 
\begin{equation}
\label{eq:chi-tilde-w=1}
\|\chi \tilde{w} \|_{L^2} = 1 + o(\h).
\end{equation}
In order to conclude, it suffices to prove that, for some $ \delta > 0 $, and~$ \h $ sufficiently small,
\begin{equation}
\label{eq:observability-E-gg-h-z}
\|1_{z>0} \chi\tilde{w}\|_{L^2} \ge \delta.
\end{equation}
Let $\tilde{\chi}\in C^{\infty}_{c}(\R)$ be such that $\chi\tilde{\chi}=\chi$, then $\chi \tilde{w}$ satisfies the equation
\begin{equation}
\label{eq:Harmonic-Oscillator-E-gg-h-truncation}
\begin{split}
(-\h^{2}\partial_{z}^{2}+\tilde{\chi}V_{E})(\chi \tilde{w})
& = \chi \tilde{w} +o(\h)_{L^2} + [\h^{2}\partial_{z}^{2},\chi] \tilde{w}  \\
& = \chi \tilde{w} + o(\h)_{L^2} + 2\h^{2}\chi'\tilde{w}'+\h^{2}\chi''\tilde{w}
= \chi \tilde{w} + o(\h)_{L^2}.
\end{split}
\end{equation}
The bottom line comes from the the inhomogeneous Lithner-Agmon estimate and the fact that $\supp\chi'\cup\supp\chi''\subset\{|z|>1\}$. This equation first implies that $ \chi \tilde{w} $ is $ \h $-oscillating (see Example~\ref{example:h-oscillating}), whose $ \h $-semiclassical measure~$ \mu $ will thus not vanish, for we have~\eqref{eq:chi-tilde-w=1}. Now that $ \chi \tilde{w} $ is supported in $ \supp \chi $, we have evidently,
\begin{equation*}
\supp \mu \subset \supp \chi \times \R_{\zeta}.
\end{equation*}
By the fact that $ \tilde{\chi}(z) V_E(z) \to \tilde{\chi}(z) z^2 $ in~$ C_c^\infty(\R) $, and Corollary~\ref{cor:Semiclassical-Measure-Modified}, we see that
\begin{equation*}
\supp \mu \subset \{(z,\zeta) : \zeta^2 + \tilde{\chi}(z) z^2 = 1 \}.
\end{equation*}
Combing the results above, 
\begin{equation*}
\supp \mu \subset \supp \chi \times \R_{\zeta} \cap \{(z,\zeta) : \zeta^2 + \tilde{\chi}(z) z^2 = 1 \} \subset \{(z,\zeta) : \zeta^2 + z^2 = 1\},
\end{equation*}
because $ \tilde{\chi} = 1 $ on $ \supp \chi $. Moreover $ \mu $ is invariant by the Hamiltonian flow generated by the Hamiltonian vector field 
\begin{equation*}
H_{\zeta^2+\tilde{\chi}(z)z^2} = (-2\zeta,2\tilde{\chi}(z)z + \tilde{\chi}'(z)z^2),
\end{equation*}
which, when restricted to $ \supp \mu $, is $ (-2\zeta,2z) $, and generates the rotation of the circle $ \zeta^2 + z^2 = 1 $. Therefore~\eqref{eq:observability-E-gg-h-z} must be satisfied, because otherwise $ \mu|_{z>0} = 0 $, and by the invariance under flow, $ \mu = 0 $, which is a contradiction.
\end{proof}

\appendix

\section{Semiclassical Measure}

\label{sec:SDM}

In this section we recall some basic properties of semiclassical measures. For details we refer to~\cite{Gerard-Semiclassical-Measure,G-L,Lions-Wigner,Burq-Bourbaki,Zworski}. In what follows $ (M,g) $ will either be the flat Euclidean space $ \R^d $ or a compact Riemannian manifold without boundary. The theory of semiclassical measure works for general Riemannian manifolds, we strict ourselves to these simple cases so that we can give some simple proofs for some of the following results, which already satisfies our needs.

\begin{thm}
\label{thm:Semiclassical-Measure}
Let~$u(h)$ ($0<h<h_{0}$) be bounded in $L^{2}(M)$. Then there exists a sequence $h_{n}\to 0$ and a positive Radon measure~$\mu$ on~$T^{*}M$ (which is called an $ h $-semiclassical measure of $u(h)$) such that for all $a \in C_{c}^{\infty}(T^{*}M)$,
\begin{equation*}
\lim_{n\to\infty} (a(x,h_n D) u(h_n), u(h_n))_{L^{2}(M)} =\int_{T^{*}M} a(x,\xi)\,\d\mu(x,\xi).
\end{equation*}
\end{thm}

\begin{rem}
We call $ \mu $ an $ h $-semiclassical measure to emphasize the importance of the parameter~$ h $, for different parameter can be used. When there is no ambiguity, we simply call~$ \mu $ a semiclassical measure.
\end{rem}

\begin{rem}
\label{remark:pure}
When there is no need to extract a subsequence, we say~$u(h)$ is pure, and~$ \mu $ is ``the'' semiclassical measure of~$u(h)$. 
\end{rem}

We also need the following corollary.
\begin{cor}
\label{cor:Semiclassical-Measure-Modified}
Let~$u(h)$ ($0<h<h_{0}$) be pure, with semiclassical measure~$ \mu $. Suppose~$\{ a_{n}\}_{n}$ and~$a$ are functions in $C_{c}^{\infty}(T^{*}M)$ such that $a_{n}\to a$ in $C^{\infty}(T^{*}M)$, then
\begin{equation*}
\lim_{n\to\infty} (a(x,h_n D) u(h_{n}), u(h_{n}))_{L^{2}(M)} =\int_{T^{*}M} a(x,\xi)\,\d\mu.
\end{equation*}
\end{cor}
This corollary is a simple consequence of Theorem~\ref{thm:Semiclassical-Measure} and the following $ L^2 $-estimate (we refer to~\cite{Zworski}, Theorem~5.1) that, for some $ N > 0 $, and all $ a \in C_c^\infty(T^*M) $,
\begin{equation*}
\|{a(x,hD)}\|_{L^{2}\to L^{2}}
\le C\sup_{|\alpha|\le Nd}h^{|\alpha|/2}\|\partial_{x,\xi}^{\alpha}a\|_{L^{\infty}}.
\end{equation*}

\begin{thm}
\label{thm:prop-sm}
Let~$u(h)$ ($0<h<h_{0}$) be pure, with semiclassical measure~$ \mu $. Let $ p \in S^m(T^*M) $ (where~$ S^m(T^*M) $ is the H\"{o}rmander class, see~\cite{Hormander}).
\begin{align*}
p(x,hD) u(h)  = o(1)_{L^2} \quad  & \Rightarrow  \quad  \supp \mu  \subset \{p = 0\} 
\\
p(x,hD) u(h)  = o(h)_{L^2} \quad  & \Rightarrow \quad  H_p \mu  = 0,
\end{align*}
where $ H_p = (-\frac{\partial H}{\partial_\xi} \frac{\partial}{\partial_x},\frac{\partial H}{\partial_x} \frac{\partial}{\partial_\xi}) $.
\end{thm}

\begin{rem}
By consequence of Corollary~\ref{cor:Semiclassical-Measure-Modified}, Theorem~\ref{thm:prop-sm} could be improved a little bit. Instead of a single symbol~$ p $, we can consider a sequence of symbols~$ \{p_n\}_{n\ge 0} \subset S^m(T^*M) $, such that $ p_n \to p $ in $ C^\infty_{\loc}(M) $. Then
\begin{align*}
p_n(x,h_nD) u(h_n)  = o(1)_{L^2} \quad  & \Rightarrow  \quad  \supp \mu  \subset \{p = 0\} 
\\
p_n(x,h_nD) u(h_n)  = o(h_n)_{L^2} \quad  & \Rightarrow \quad  H_p \mu  = 0.
\end{align*}
\end{rem}

For a pure sequence~$ u(h) $, even if its semiclassical measure $ \mu = 0 $, we do not generally have $ u(h) \to 0 $ strongly in~$ L^2(M) $. However, this is the case if~$ u(h) $ is in addition $ h $-oscillating.

\begin{defn}
A sequence~$ u(h) $ is called $ h $-oscillating if for some $ \chi \in C^\infty(\R) $ such that $ 0 \le \chi \le 1 $, $ \chi = 0 $ in a neighborhood of the origin, and $ \chi = 1 $ outside a neighborhood of the origin, then
\begin{equation*}
\lim_{R\to+\infty} \limsup_{h\to 0} \| \chi(-h^2\Delta/R) u(h) \|_{L^2} = 0.
\end{equation*}
\end{defn}

\begin{prop}
\label{prop:h-oscillating-imply-Ltwo-convergence}
Let $ u(h) $ be a pure and $ h $-oscillating sequence with vanishing semiclassical measure, then $ u(h) \to 0 $ strongly in $ L^2(M) $.
\end{prop}
\begin{proof}
Let $ \chi $ be chose as in the definition above. Write $ \chi_R(\cdot) = \chi(\cdot/R) $ for simplicity, and decompose
\begin{equation*}
\|u(h)\|_{L^2}^2 = \big((1-\chi^2_R(-h^2\Delta)) u(h),u(h)\big)_{L^2} + \big(\chi^2_R(-h^2\Delta) u(h),u(h)\big)_{L^2}.
\end{equation*}
Observe that $ (1-\chi^2_R(-h^2\Delta)) $ is a semiclassical pseudodifferential operator with principal symbol $ 1-\chi^2_R(g^{-1}_x(\xi,\xi)) \in C_c^\infty(T^*M) $, therefore, since $ \mu = 0 $,
\begin{equation*}
\lim_{h\to 0} \big((1-\chi^2_R(-h^2\Delta)) u(h),u(h)\big)_{L^2} = \int_{T^{*}M} \Big( 1-\chi^2_R(g^{-1}_x(\xi,\xi)) \Big) \d\mu(x,\xi) = 0.
\end{equation*}
While for the second term, by our hypothesis of $ h $-oscillation,
\begin{equation*}
\lim_{R\to+\infty}\limsup_{h\to 0} \big(\chi^2_R(-h^2\Delta) u(h),u(h)\big)_{L^2} 
= \lim_{R\to+\infty}\limsup_{h\to 0} \|\chi_R(-h^2\Delta) u(h)\|_{L^2}^2
= 0.
\end{equation*}
Combine these two limit behaviors,
\begin{equation*}
\limsup_{h\to 0} \|u(h)\|_{L^2}^2 = \lim_{R\to+\infty}\limsup_{h\to 0} \|u(h)\|_{L^2}^2 = 0.
\end{equation*}
\end{proof}

\begin{ex}
\label{example:h-oscillating}
Suppose that~$ u(h) $ is a pure sequence satisfying
\begin{equation*}
(-h^2\Delta + V) u(h) = o(1)_{L^2},
\end{equation*}
where $ V \in C_c^\infty(M) $ then $ u(h) $ is $ h $-oscillating.
Indeed, by adding to~$ V $ some constant, we may assume that $ V \ge 1 $. So that $ -h^2\Delta + V $ is a self-adjoint operator with uniformly (in~$ h $) bounded resolvent 
$ \|(-h^2\Delta + V)^{-1}\|_{L^2\to L^2} \le 1 $. After this modification~$ u(h) $ satisfies an equation of the form
\begin{equation*}
(-h^2\Delta + V) u(h) = E u(h) + o(1)_{L^2},
\end{equation*}
for some constant $ E \in \R $. Denote $ \psi(z) = z^{-1} \chi(z) \in C_c^\infty(\R) $, and $ \psi_R(z) = \psi(z/R) $, then apply each side of the equation above by $ R^{-1} \psi_R(-h^2\Delta + V) $, we obtain
\begin{equation*}
\chi_R(-h^2\Delta + V) u(h) = E R^{-1} \psi_R(-h^2\Delta + V) u(h) + o(1)_{L^2}.
\end{equation*}
Now that~$ M $ is either compact or Euclidean, we have a uniform elliptic estimate $ g^{-1}_x(\xi,\xi) \gtrsim |\xi|^2 $, whence for~$ R $ sufficiently large and for $ (x,\xi) \in T^*M $,
\begin{equation*}
\chi_{R/2}(g^{-1}_x(\xi,\xi)) \chi_R(g^{-1}_x(\xi,\xi) + V(x)) = \chi_{R/2}(g^{-1}_x(\xi,\xi)).
\end{equation*}
Therefore, apply the equation above by $ \chi_{R/2}(-h^2\Delta) $, and by a symbolic calculus, we have
\begin{equation*}
\chi_{R/2}(-h^2\Delta) u(h) = E R^{-1} O(1)_{L^2} + o(1)_{L^2},
\end{equation*}
which implies that, for~$ R $ sufficiently large,
\begin{equation*}
\limsup_{h\to 0} \|\chi_{R/2}(-h^2\Delta) u(h)\|_{L^2} = O(R^{-1}).
\end{equation*}
We also remark that this argument works when~$ V $ depends on~$ h $, but remains bounded in~$ C_c^\infty(M) $.
\end{ex}

\vspace{\baselineskip}
\parbox{5in}{
HUI ZHU

\textsc{Laboratoire de Mathématiques d’Orsay}

\textsc{Université Paris-Sud, CNRS, Université Paris-Saclay}

\textsc{91405 Orsay, France}

\textit{E-mail}: \texttt{hui.zhu@math.u-psud.fr}
}


\begin{thebibliography}{99}

\bibitem{Agmon} Agmon, S. (2014). Lectures on Exponential Decay of Solutions of Second-Order Elliptic Equations: Bounds on Eigenfunctions of N-Body Schr\"{o}dinger Operations.(MN-29). Princeton University Press.

\bibitem{A-G} Alinhac, S., Gérard, P. Pseudo-differential operators and the Nash-Moser theorem, translated from the 1991 French original by Stephen S. Wilson. Graduate Studies in Mathematics, 82.

\bibitem{Arnold} Arnold, V.~I. (2006). Ordinary differential equations. Translated from the Russian by Roger Cooke. Second printing of the 1992 edition. Universitext.

\bibitem{B-L-R}Bardos, C., Lebeau, G., Rauch, J. (1992). Sharp sufficient conditions for the observation, control, and stabilization of waves from the boundary. SIAM journal on control and optimization, 30(5), 1024-1065.

\bibitem{Beekmann}Beekmann, B. (1990). Eigenfunctions and Eigenvalues on Surfaces of Revolution. Results in Mathematics, 17(1-2), 37-51. 

\bibitem{Besse}Besse, A.~L. (2012). Manifolds all of whose geodesics are closed (Vol. 93). Springer Science \& Business Media.

\bibitem{Burq-1}Burq, N. (2002). Semi-classical estimates for the resolvent in nontrapping geometries. International Mathematics Research Notices, 2002(5), 221-241. 

\bibitem{Burq-Bourbaki}Burq, N. (1997). Mesures semi-classiques et mesures de défaut. Séminaire Bourbaki, 39, 167-195. 

\bibitem{B-G}Burq, N., Gérard, P. (2002). Contrôle optimal des equations aux derivées partielles. Ecole polytechnique, Département de mathématiques. 

\bibitem{B-G-Stabilization-Wave-Tori} Burq, N. Gérard, P. (2017). Second Microlocalization and Stabilization of Damped Wave Equations on Tori. In preparation.

\bibitem{D-S}Dimassi, M., \& Sj\"{o}strand, J. (1999). Spectral asymptotics in the semi-classical limit (No. 268). Cambridge university press.

\bibitem{D-G-Zoll} Duistermaat, J.~J., \& Guillemin, V. W. (1975). The Spectrum of Positive Elliptic Operators and Periodic Bicharacteristics. Inventiones Mathematicae, 29, 39-80.

\bibitem{Fedoryuk}Fedoryuk, M.~V. (1993). Asymptotic analysis: linear ordinary differential equations. Springer Verlag.

\bibitem{Gerard-MDM} Gérard, P. (1991). Microlocal defect measures. Communications in Partial differential equations, 16(11), 1761-1794.

\bibitem{Gerard-Semiclassical-Measure} Gérard, P. (1991). Mesures semi-classiques et ondes de Bloch. Séminaire Équations aux dérivées partielles (Polytechnique), 1-19.

\bibitem{G-L} Gérard, P., \& Leichtnam, É. (1992). Ergodic properties of eigenfunctions for the Dirichlet problem. Université de Paris-sud, Département de mathématiques.

\bibitem{Helffer} Helffer, B. (2006). Semi-classical analysis for the Schrödinger operator and applications (Vol. 1336). Springer.

\bibitem{Hormander} H\"{o}rmander, L. (2007). The analysis of linear partial differential operators III: Pseudo-differential operators (Vol. 274). Springer Science \& Business Media.

\bibitem{Ingham} Ingham, A.~E. (1936). Some trigonometrical inequalities with applications to the theory of series. Mathematische Zeitschrift, 41(1), 367-379.

\bibitem{Lebeau}Lebeau, G. (1996). Equation des ondes amorties (pp. 73-109). Springer Netherlands. 

\bibitem{Lions} Lions, J.~L. (1988). Exact controllability, stabilization and perturbations for distributed systems. SIAM review, 30(1), 1-68.

\bibitem{Lions-Wigner} Lions, P.~L., \& Paul, T. (1993). Sur les mesures de Wigner. Revista matemática iberoamericana, 9(3), 553-618.

\bibitem{Lithner} Lithner, L. (1964). A theorem of the Phragmén-Lindel\"{o}f type for second-order elliptic operators. Arkiv f\"{o}r Matematik, 5(3), 281-285.


\bibitem{Melrose-Sjostrand} Melrose, R.~B., \& Sj\"{o}strand, J. (1978). Singularities of boundary value problems. I. Communications on Pure and Applied Mathematics, 31(5), 593-617.

\bibitem{Olver}Olver, F.~W. (2014). Asymptotics and special
functions. Academic press.

\bibitem{Tartar-H-Measures} Tartar L. H-measures, a new approach for studying homogenisation, oscillations and concentration effects in partial differential equations[J]. Proceedings of the Royal Society of Edinburgh: Section A Mathematics, 1990, 115(3-4): 193-230.

\bibitem{R-T}Rauch, J., Taylor, M. (1975). Decay of solutions to nondissipative hyperbolic systems on compact manifolds. Communications on Pure and Applied Mathematics, 28(4), 501-523. 

\bibitem{S-W}Stein, E.~M., Weiss, G. Introduction to Fourier analysis on Euclidean spaces. 1971. Princeton, New Jersey. 

\bibitem{Zuazua}Zuazua, E. (2007). Controllability and observability of partial differential equations: some results and open problems. Handbook of differential equations: evolutionary equations, 3, 527-621. 

\bibitem{Zworski}Zworski, M. (2012). Semiclassical analysis (Vol. 138). American Mathematical Soc..

\end{thebibliography}
\end{document}